\documentclass[final]{siamltex}

\newtheorem{remark}{Remark}[section]

\usepackage{times,amsmath,amsbsy,amssymb}
\usepackage{graphicx}

\title{NEW RESIDUAL-BASED A POSTERIORI ERROR ESTIMATORS FOR LOWEST-ORDER RAVIART-THOMAS
ELEMENT APPROXIMATION TO CONVECTION-DIFFUSION-REACTION
EQUATIONS\thanks{This work was supported in part by The Natural
Science Foundation of Chongqing city under Grant No. CSTC,
2010BB8270, The Education Science Foundation of Chongqing
(KJ120420),
 National Natural Science Foundation of China (11171239), The Project-sponsored
by Scientific Research Foundation for the Returned Overseas Chinese Scholars and
 Open Fund of  Key Laboratory of Mountain Hazards and Earth Surface Processes, CAS.}}

\author{Shaohong Du\thanks{School of Science, Chongqing Jiaotong University,
                           Chongqing 400047, China, ({\it
                           dushhong@gmail.com}).}
\and Xiaoping XIE\thanks{School of Mathematics, Sichuan University,
Chengdu 610064, China ({\it xpxiec@gmail.com}).
}}

\begin{document}

\maketitle
\begin{small}
  {\bf{Abstract.} \rm{A new technique of residual-type a posteriori error analysis
   is developed for the lowest-order Raviart-Thomas mixed finite element
  discretizations of convection-diffusion-reaction equations in two- or three-dimension.
 Both centered mixed scheme and upwind-weighted mixed scheme are
  considered. The a posteriori error estimators, derived for the
  stress variable error plus
  scalar displacement error in $L^{2}$-norm, can be directly computed
  with the solutions of the mixed schemes without any
  additional cost, and are robust with respect to the  coefficients in the equations.
  Local efficiency
  dependent on local variations in coefficients is
  obtained without any saturation assumption, and holds from the
  cases where convection or reaction is not present to convection-
  or reaction-dominated problems. The main tools of analysis are the
  postprocessed approximation of scalar displacement, 
  abstract error estimates, and the property of  modified Oswald
  interpolation. Numerical
experiments are reported to support our theoretical results and to
show the competitive behavior of the proposed posteriori error
estimates.}}
\end{small}

\begin{it} Key words.\end{it} convection-diffusion-reaction
equation, centered mixed scheme,
upwind-weighted mixed scheme,  postprocessed approximation, a posteriori
error estimators

\begin{it} AMS subject classifications.\end{it} 65N15, 65N30, 76S05

\pagestyle{myheadings} \thispagestyle{plain} \markboth{S. DU AND X.
XIE}{A NEW POSTERIORI ERROR ESTIMATORS FOR CONVECTION-DIFFUSION
EQUATIONS}

\section {Introduction}
Let $\Omega\subset\mathbb{R}^{d}$ be a bounded polygonal
or polyhedral  domain  in
${\mathbb{R}}^{d}, d=2\ {\rm or}\ 3$. We consider the following
homogeneous Dirichlet boundary value problem for the
convection-diffusion-reaction equations:
\begin{equation}\label{convection-diffufsion-equations1}
 \left \{ \begin{array}{ll}
    -\nabla\cdot(S\nabla p)+\nabla\cdot(p{\bf w})+rp  =f \quad   \mbox{in}\ \ \Omega,\\
    \hspace{41mm} p=0  \quad \mbox{on}\ \partial{\Omega},
 \end{array}\right.
\end{equation}
where $S\in L^{\infty}(\Omega;{\mathbb{R}}^{d\times d})$ denotes an
inhomogeneous and anisotropic diffusion-dispersion tensor, ${\bf
w}$ is a (dominating) velocity field, $r$ a reaction function, $f$ a
source term. The choice of boundary conditions is made for ease of
presentation, since similar results are valid for other boundary
conditions. This type of equations arise in many chemical and biological settings.
For instance, in hydrology these equations  govern the
transport and degradation of adsorbing contaminants and microbe-nutrient systems
in groundwater.

Reliable and efficient a posteriori error estimators are an indispensable tool for
  adaptive algorithms.
For second-order elliptic problems without convection term, the theory of a
 posteriori error estimation has reached a degree of maturity for  finite elements of conforming,
 nonconforming and mixed types; see [1-9, 11-14, 18, 20, 22-23, 27, 31-33] and
the references therein.  For convection-diffusion(-reaction)
problems, on the contrary, the theory is still under development.

The mathematical analysis of robustness of a-posteriori estimators
for the convection-diffusion-reaction equations was first addressed
by Verf\"{u}rth \cite{Verfurth-4} in the singular perturbation case,
namely $S=\varepsilon I$ with $I$ the identical matrix and
$0<\varepsilon\ll1$.  The proposed estimators  for the standard
Galerkin approximation and the SUPG disctetization give global upper
and local lower bounds on the error measured in the energy norm,
and are  robust when the P$\acute{e}$clet number becomes small.  In
\cite{Verfurth-5} {\it Verf\"{u}rth}  improved the
results of \cite{Verfurth-4} in the sense that the derived estimates
are fully robust with respect to convection dominance  and uniform
with respect to the size of the zero-order reaction term. {\it
Sangalli} \cite{Sangalli2001} developed an a posteriori estimator
for the residual-free bubbles methods applied to
convection-diffusion problems. Later he  presented a
residual-based a posteriori estimator for the one-dimensional
convection-diffusion-reaction model problem \cite{San08}. In
\cite{Kunert2003} {\it Kunert} carried out a posteriori error
estimation for the SUPG approach to a singularly perturbed
convection-diffusion problem  on anisotropic meshes. One may also
refer to \cite{Ohlberger,Ohlberger1} for a posteriori error
estimation in the framework of finite volume approximations.

For the convection-diffusion-reaction model
(\ref{convection-diffufsion-equations1}), following an idea of postprocessing 
 in \cite{Lovadina and Stenberg} {\it Vohral\'{\i}k}
\cite{Vohralik1} established residual
a posteriori error estimates for lowest-order Raviart-Thomas mixed
finite element discretizations  on simplicial meshes.  Global upper
bounds and local lower bounds for the postprocessed approximation
error, $p-\tilde{p}_h$, in the energy norm were derived with $\tilde{p}_h$  the postprocessed approximation to the finite element solution $p_h$, and the local efficiency of
the estimators was shown to depend only on local variations in the
coefficients and on the local P$\acute{e}$clet number. 
 Moreover,
the developed general framework allows for asymptotic exactness and
full robustness with respect to inhomogeneities and anisotropies.

In this paper,  we develop 
a new technique for  residual-based a
posteriori estimation of the lowest-order Raviart-Thomas mixed
finite element schemes (centered mixed scheme and upwind-mixed
scheme) over both the stress error, ${\bf u}-{\bf u}_h$,  and the displacement error, $p-p_h$, of the mixed finite element solutions $({\bf u}_h, p_h)$ for
the problem (\ref{convection-diffufsion-equations1}) with ${\bf u} :=-S\nabla p$. The derived reliability results    are robust with respect to the
coefficients. Local efficiency
  dependent only on local variations in the coefficients is
  obtained without any saturation assumption,  holds  for
the convection or reaction dominated equations. 
Compared with the standard
analysis to the diffusion equations, 
our analysis avoids, by using the postprocessed approximation $\tilde{p}_h$ as a transition, Helmholtz decomposition
of  stress variables and  dual arguments of  displacement
error in $L^{2}$-norm, and then does not need any weak
regularity assumption on the diffusion-dispersion tensor.      We note that although being employed in our analysis, the postprocessed displacement approximation and its modified Oswald interpolation are not  involved in
our estimators.

The rest of this paper is organized as follows. In Section 2 we give
notations, assumptions of   data, and the weak problem. We introduce
in Section 3 the mixed finite element schemes (include the centered
and upwind-weighted mixed scheme) and the post-processed techniques.
Section 4 includes  the main results. Section 5 collects some
preliminary results and   remarks. Section 6 and 7 analyze
respectively   the a posteriori error estimates and  the local
efficiency.  Finally, we present several numerical examples in
Section 8 to test our estimators.

\section{Notations, assumptions and weak problem}\ For a domain $A\subset\mathbb{R}^{d}$, we denote by
$L^{2}(A)$ and ${\bf L}^{2}(A) =:(L^{2}(A))^{d}$ the   spaces   of
square-integrable functions,
by $(\cdot,\cdot)_{A}$ the $L^{2}(A)$ or ${\bf L}^{2}(A)$ inner
product, by $||\cdot||_{A}$ the associated norm, and by $|A|$   the Lebesgue measure of $A$.  Let
$H^{k}(A)$ be the usual Sobolev space consisting of functions defined on $A$ with all
derivatives of order up
to  $k$ square-integrable;  $H_0^1(A):=\{v\in H^{1}(A):\ v|_{\partial A}=0\}$, ${\bf H}({\rm div},A):=\{{\bf
v}\in {\bf L}^{2}(A):{\rm div}\ {\bf v}\in L^{2}(A)\}$.
$<\cdot,\cdot>_{\partial A}$ denotes $d-1$-dimensional inner
product on $\partial A$ for the duality paring between
$H^{-1/2}(\partial A)$ and $H^{1/2}(\partial A)$.

Let $\mathcal{T}_{h}$ be a shape regular triangulation  in the sense
of \cite{Ciarlet} which satisfies the angle condition, namely there
exists a constant $c_{0}$ such that for all $K\in\mathcal{T}_{h}$ with $h_{K} :={\rm diam}(K)$,
\begin{equation*}
c_{0}^{-1}h_{K}^{d}\leq|K|\leq c_{0}h_{K}^{d}.
\end{equation*}
Let $C_{Q},c_{Q}$ be positive constants dependent only on a quantity
$Q$, and $c_{i} (i=1,2,\cdots)$   positive constants  determined
only by  the shape regularity parameter, $c_{0}$, of
$\mathcal{T}_{h}$. We denote by $\varepsilon_{h}$ the set of element
sides in $\mathcal{T}_{h}$, by $\varepsilon_{h}^{{\rm int}}$ and
$\varepsilon_{h}^{{\rm ext}}$ the sets of all interior and exterior
sides of $\mathcal{T}_{h}$, respectively. For $K\in\mathcal{T}_{h}$,
denote by $\varepsilon_{K}$ the set of sides of $K$, especially by
$\varepsilon_{K}^{\rm ext}$ the set of the boundary sides of $K$.
Furthermore, we denote by $\omega_{\sigma}$ and $
\tilde{\omega}_{\sigma}$ the union of all elements in
$\mathcal{T}_{h}$ sharing a side $\sigma$ and the union of all
elements sharing at least one point of $\sigma$, respectively. For
an element $K\in\mathcal{T}_{h}$ the set $\tilde{\omega}_{K}$ is
defined analogously. We  also use the "broken Sobolev space"
$H^{1}(\bigcup\mathcal{T}_{h}) :=\{\varphi\in
L^{2}(\Omega):\varphi|_{K}\in H^{1}(K),\forall
K\in\mathcal{T}_{h}\}$, and denote by $[v]|_{\sigma}
:=(v|_{K})|_{\sigma}-(v|_{L})|_{\sigma}$ the jump of $v\in
H^{1}(\bigcup\mathcal{T}_{h})$ over an interior side $\sigma :=\bar{K}\cap
\bar{L}$ of diameter $h_{\sigma}: ={\rm diam}(\sigma)$, shared by the two
neighboring (closed) elements $K,L\in\mathcal {T}_{h}$. Especially,
$[v]|_{\sigma} :=(v|_{K})|_{\sigma}$ if
$\sigma\in\varepsilon_{K}^{\rm ext}$.

We consider $d=2,3$ simultaneously and let $m :=1$ if $d=2$ and $m
:=3$ if $d=3$. The Curl of a function $\psi\in H^{1}(\Omega)^{m}$ is
defined by
\begin{equation*}
{\rm Curl}\psi :=(-\partial \psi/\partial x_2, \partial \psi/\partial x_1)\ {\rm if\ d=2}\ \ {\rm and}\ \ {\rm
Curl}\psi :=\nabla\times\psi\ \ {\rm if\ d=3},
\end{equation*}
where $\times$ denotes
the usual vector product of two vectors in $\mathbb{R}^{3}$. Given a unit normal vector ${\bf n}=(n_1,n_2)$
along the side $\sigma$, we define the tangential component of a
vector ${\bf v}\in\mathbb{R}^{d}$ 
by
\begin{equation*}
\gamma_{{\bf t}_{\sigma}}({\bf v}) :=\left \{ \begin{array}{ll}
{\bf v}\cdot (-n_2,n_1) \quad   \mbox{if}\ \ d=2,\\
{\bf v}\times{\bf n} \quad   \mbox{if}\ \ d=3.
 \end{array}\right.
\end{equation*}

We note that throughout the paper, the local versions of
differential operators $\nabla, {\rm curl}$ are understood in the
distribution sense, namely, ${\rm
curl}_{h}:H^{1}(\bigcup\mathcal{T}_{h})^{d}\rightarrow
L^{2}(\Omega)$ and
$\nabla_{h}:H^{1}(\bigcup\mathcal{T}_{h})\rightarrow
L^{2}(\Omega)^{d}$ are defined with ${\rm curl}_{h}{\bf v}|_{K}
:={\rm curl}({\bf v}|_{K}) $ and $\nabla_{h}{\bf w}|_{K}
:=\nabla({\bf w}|_{K}) $  for all $\ K\ \in\mathcal{T}_{h}$.

We need in our analysis the following inequalities, Poincar\'{e} inequality and
Friedrichs inequalities  \cite{Bebendorf,Payne;Weinberger}: for $K\in\mathcal{T}_{h}$
and $\varphi\in H^{1}(K)$,
\begin{equation}\label{Poicare inequality}
||\varphi-\varphi_{K}||_{K}^{2}\leq
C_{P,d}h_{K}^{2}||\nabla\varphi||_{K}^{2},
\end{equation}
\begin{equation}\label{Friedrichs inequality}
(\varphi_{K}-\varphi_{\sigma})^{2}\leq
\frac{3dh_{K}^{2}}{|K|}||\nabla\varphi||_{K}^{2},\ \
||\varphi-\varphi_{\sigma}||_{K}^{2}\leq3dh_{K}^{2}||\nabla\varphi||_{K}^{2}.
\end{equation}
Here
$\varphi_{K} :=(1,\varphi)_{K}/|K|$ and $\varphi_{\sigma} :=<1,\varphi>_{\sigma}/|\sigma|$ denote the integrable means of $\varphi$ over $K$ and over $\sigma\in\varepsilon_{K}$, respectively.
The constant $C_{P,d}$ can be evaluated as $d/\pi$ for a
simplex by using its convexity.

Following \cite{Vohralik1}, we suppose that there exists an original triangulation
$\mathcal{T}_{0}$ of $\Omega$ such that    data of the problem
($\ref{convection-diffufsion-equations1}$) are given in the
following way.\\
\textbf{Assumptions of data }:\\
{\it (D1)\ $S_{K} :=S|_{K}$ is a constant, symmetric, and uniformly
positive definite tensor such that $c_{S,K}{\bf v}\cdot{\bf v}\leq S_{K}{\bf
v}\cdot{\bf v}\leq C_{S,K}{\bf v}\cdot{\bf v}$ holds for all ${\bf v}\in\mathbb{R}^{d}$ and all
$K\in\mathcal{T}_{0}$ with $c_{S,K},C_{S,K}>0$;\\
(D2)\ ${\bf w}\in RT_{0}(\mathcal {T}_{0})$ (cf, Section 3 below)
such that $|{\bf w}|_{K}|\leq C_{{\bf w},K} $ holds  for
all
$K\in\mathcal{T}_{0}$ with $C_{{\bf w},K}\geq0$;\\
(D3)\ $r_{K} :=r|_{K}$ is a constant for all
$K\in\mathcal{T}_{0}$;\\
(D4)\ $c_{{\bf w},r,K} :=\frac{1}{2}\nabla\cdot{\bf w}|_{K}+r_{K}\geq0$ and
$C_{{\bf w},r,K} :=|\nabla\cdot{\bf w}|_{K}+r_{K}|$ 
 for all $K\in\mathcal{T}_{0}$;\\
(D5)\ $f|_{K}$ is a polynomial for each
$K\in\mathcal{T}_{0}$;\\
(D6)\ if $c_{{\bf w},r,K}=0$, then $C_{{\bf w},r,K}=0$.}

As pointed out in \cite{Vohralik1}, all the assumptions are made for the
sake of simplicity and are usually satisfied in practice. If 
data do not satisfy these assumptions, we may employ the
interpolation or projection of data with additional occurrence of
data oscillation.

Finally we show the weak problem of the model (\ref{convection-diffufsion-equations1}): Find $p\in
H_{0}^{1}(\Omega)$ such that
\begin{equation}\label{continuous problem 2}
\mathcal{B}(p,\varphi)=(f,\varphi)\ \ \ {\rm for\ all}\ \ \varphi\in
H_{0}^{1}(\Omega).
\end{equation}
Here the bilinear form 
\begin{equation*}
\mathcal{B}(p,\varphi) :=\sum\limits_{K\in\mathcal{T}_{h}}\{(S\nabla
p,\nabla\varphi)_{K}+(\nabla\cdot(p{\bf
w}),\varphi)_{K}+(rp,\varphi)_{K}\},p,\varphi\in
H^{1}(\bigcup\mathcal{T}_{h}),
\end{equation*}
and  $\mathcal{T}_{h}$ is a refinement of $\mathcal{T}_{0}$.
 We define as following an  energy (semi) norm corresponding to the bilinear form $\mathcal{B}$:
\begin{equation*}
|||\varphi|||_{\Omega}^{2}
:=\sum\limits_{K\in\mathcal{T}_{h}}|||\varphi|||_{K}^{2},\
|||\varphi|||_{K}^{2} :=(S\nabla\varphi,\nabla\varphi)_{K}+c_{{\bf
w},r,K}||\varphi||_{K}^{2},\varphi\in H^{1}(\bigcup\mathcal{T}_{h}).
\end{equation*}

We note that the weak problem (\ref{continuous problem 2}) admits a
unique solution under the Assumptions (D1)-(D6) \cite{Vohralik1}.

\section{Mixed finite element schemes and postprocessing}
Since it is of interest in many applications, the stress variable
${\bf u} :=-S\nabla p$ are usually approximated by using the mixed
finite elements for the problem
(\ref{convection-diffufsion-equations1}). We introduce in this
section the centered and upwind-weighted mixed finite element
schemes, and show the postprocessed techniques presented by {\it
Vohral\'{\i}k} in \cite{Vohralik1}.

We  define the lowest order Raviart-Thomas finite
element and piecewise constant space respectively as following:
$$
RT_{0}(\mathcal {T}_{h}) :=\left\{ \begin{array}{c} {\bf q}_{h}\in
{\bf H}({\rm div},\Omega):\ \forall K\in\mathcal {T}_{h},\  \exists
{\bf a}\in\mathbb{R}^{d},\  \exists b\in\mathbb{R}, \\ {\rm such\
that}\ {\bf q}_{h}({\bf x})={\bf a}+b{\bf x},{\rm for\ all}\ {\bf
x}\in K.
\end{array}\right\},
$$
$$
P_{0}({\mathcal {T}_{h}}) :=\{v_{h}\in L^{\infty}(\Omega):\ \forall
K\in\mathcal {T}_{h},\  v_{h}|_{K}\in P_{0}(K) \}.
$$
Here ${\bf n}$ is the unit outer normal vector along $\sigma\in\varepsilon_{h}$, and $P_{0}(K)$
denotes the set of constant functions on each $K\in\mathcal{T}_{h}$. We note
that $\nabla\cdot (RT_{0}(\mathcal{T}_{h}))\subset
P_{0}({\mathcal {T}_{h}})$.

 The centered mixed finite element scheme   \cite{Douglas,Vohralik1} reads as: Find $({\bf u}_{h},p_{h})\in
RT_{0}(\mathcal {T}_{h})\times P_{0}({\mathcal {T}_{h}})$ such that
\begin{equation}\label{centered mixed scheme}
(S^{-1}{\bf u}_{h},{\bf v}_{h})_{\Omega}-(p_{h},\nabla\cdot{\bf
v}_{h})_{\Omega}=0\ \ \ {\rm for\ all}\ {\bf v}_{h}\in
RT_{0}(\mathcal {T}_{h}),
\end{equation}
\begin{equation}\label{centered mixed scheme 2}
(\nabla\cdot{\bf u}_{h},\varphi_{h})_{\Omega}-(S^{-1}{\bf
u}_{h}\cdot{\bf w},\varphi_{h})_{\Omega}+((r+\nabla\cdot{\bf
w})p_{h},\varphi_{h})_{\Omega}=(f,\varphi_{h})_{\Omega}\ \  {\rm for\ all}\ \ \varphi_{h}\in P_{0}({\mathcal
{T}_{h}}).
\end{equation}

The upwind-weighted mixed finite element scheme \cite{Dawson,Vohralik1}  reads as: Find
$({\bf u}_{h},p_{h})\in RT_{0}(\mathcal {T}_{h})\times
P_{0}({\mathcal {T}_{h}})$ such that
\begin{equation}\label{upwind-weighted scheme 1}
(S^{-1}{\bf u}_{h},{\bf v}_{h})_{\Omega}-(p_{h},\nabla\cdot{\bf
v}_{h})_{\Omega}=0\ \ \ {\rm for\ all}\ {\bf v}_{h}\in
RT_{0}(\mathcal {T}_{h}),
\end{equation}
\begin{equation}\label{upwind-weighted scheme 2}
(\nabla\cdot{\bf
u}_{h},\varphi_{h})_{\Omega}+\sum\limits_{K\in\mathcal{T}_{h}}
\sum\limits_{\sigma\in\varepsilon_{K}}\hat{p}_{\sigma}w_{K,\sigma}\varphi_{K}
+(rp_{h},\varphi_{h})_{\Omega}=(f,\varphi_{h})_{\Omega}\ \ {\rm for\ all}\ \ \varphi_{h}\in P_{0}({\mathcal
{T}_{h}}),
\end{equation}
where $w_{K,\sigma} :=<1,{\bf w}\cdot{\bf n}>_{\sigma}$ for
$\sigma\in\varepsilon_{K}$, with ${\bf n}$  the unit normal
vector of   $\sigma$, outward to $K$, $\varphi_{K} =(1,\varphi_{h})_K/|K|=\varphi_{h}|_{K}$ for all
$K\in\mathcal{T}_{h}$, and $\hat{p}_{\sigma}$ is  the weighted upwind value given by
\begin{equation}\label{upwind-weighted scheme 3}
 \hat{p}_{\sigma} :=\left \{ \begin{array}{ll}
  (1-\nu_{\sigma})p_{K}+\nu_{\sigma}p_{L}\ \  & \mbox{if}\;  \ \ w_{K,\sigma}\geq0,\\
 (1-\nu_{\sigma})p_{L}+\nu_{\sigma}p_{K}\ \  & \mbox{if}\; \ \
 w_{K,\sigma}<0
 \end{array}\right.
\end{equation}
when $\sigma$ is an interior side sharing by elements $K$ and $L$,
and by
\begin{equation}\label{upwind-weighted scheme 4}
 \hat{p}_{\sigma} :=\left \{ \begin{array}{ll}
  (1-\nu_{\sigma})p_{K}\ \  & \mbox{if}\;  \ \ w_{K,\sigma}\geq0,\\
 \nu_{\sigma}p_{K}\ \  & \mbox{if}\; \ \
 w_{K,\sigma}<0
 \end{array}\right.
\end{equation}
when $\sigma$ is a boundary side included in $\varepsilon_{K}$.
Here $p_K$ and $p_L$ denotes respectively the restrictions  of
$p_{h}$ over $K$ and L,   $\nu_{\sigma}\in[0,1/2]$ denotes the
coefficient of the amount of upstream weighting which may be chosen as \cite{Vohralik1}
\begin{equation}\label{upwind-weighted scheme 5}
\nu_{\sigma} :=\left \{ \begin{array}{ll}
 \min\{c_{S,\sigma}\frac{|\sigma|}{h_{\sigma}|w_{K,\sigma}|},\frac{1}{2}\}
 \  & \mbox{if}\;  \ \ w_{K,\sigma}\neq0\ {\rm and}\ \sigma\in\varepsilon_{h}^{{\rm int}},\\
\ &{\rm or\ if}\ \sigma\in\varepsilon_{h}^{{\rm ext}}\ {\rm and}\
w_{K,\sigma}>0;\\
 0\ \  & \mbox{if}\; \ \ w_{K,\sigma}=0\ {\rm or\ if}\ \sigma\in\varepsilon_{h}^{{\rm ext}}\ {\rm and}\
w_{K,\sigma}<0,
\end{array}\right.
\end{equation}
where $c_{S,\sigma}$ is the harmonic average of $c_{S,K}$ and
$c_{S,L}$ if $\sigma\in\partial K\cap\partial L$ and $c_{S,K}$
otherwise.

We now introduce the postprocessed technique in \cite{Vohralik1},
where  a postprocessed approximation $\tilde{p}_{h}$ to the
displacement $p$ is constructed which links $p_h$ and ${\bf u}_{h}$
  on each simplex in the following way:
\begin{equation}\label{postprocessed variable 1}
-S_{K}\nabla\tilde{p}_{h}|_{K}={\bf u}_{h}\ \ \ {\rm for\ all}\ \
K\in\mathcal{T}_{h},
\end{equation}
\begin{equation}\label{postprocessed variable 2}
\frac{1}{|K|}\int_{K}\tilde{p}_{h}d{\bf x}=p_{K}\ \ \ {\rm for\
all}\ \ K\in\mathcal{T}_{h}.
\end{equation}
We refer to \cite{Vohralik1}   for the existence of
$\tilde{p}_{h}$. We  note that the new quantity $\tilde{p}_{h}\in
W_{0}(\mathcal{T}_{h})$ but $\notin H_{0}^{1}(\Omega)$ (see LEMMA
6.1 in \cite{Vohralik1}), where
\begin{equation*}
\begin{array}{lll}
W_{0}(\mathcal {T}_{h}) &:=&\{\varphi\in L^{2}(\Omega):
\varphi|_{K}\in H^{1}(K)\ {\rm for\ all}\  K\in\mathcal {T}_{h}, \
<1,\varphi|_{K}-\varphi|_{L}>_{\sigma_{K,L}}\vspace{2mm}\\
\ \ &=0&{\rm for \ all}\ \ \sigma_{K,L}\in\varepsilon_{h}^{{\rm int}},\
\ <1,\varphi>_{\sigma}=0\ \ {\rm for\ all}\
\sigma\in\varepsilon_{h}^{{\rm ext}}\}.
\end{array}
\end{equation*}

\section{Main results}\ With the stress variable ${\bf u}=-S\nabla p$,
we
define  the global and local errors,  $\mathcal{E}$ and $\mathcal{E}_K$, of
the stress and displacement variables
as
\begin{equation}\label{new-1}
\mathcal{E} :=\{\sum\limits_{K\in\mathcal{T}_{h}}\mathcal{E}_{K}^2\}^{1/2}, \  \mathcal{E}_{K}^2 :=||S^{-1/2}({\bf u}-{\bf
u}_{h})||_{K}^{2}+c_{{\bf w},r,K}||p-p_{h}||_{K}^{2}.
\end{equation}

Denote respectively by $\eta_{D,K}$ and $\eta_{R,K}$ the elementwise
displacement and residual estimator with
\begin{equation}\label{new 0}
\eta_{D,K}^{2} :=c_{{\bf w},r,K}h_{K}^{2}||S^{-1}{\bf
u}_{h}||_{K}^{2},
\end{equation}
\begin{equation}\label{new 1}
\eta_{R,K}^{2} :=\alpha_{K}^{2}||f-\nabla\cdot{\bf
u}_{h}+(S^{-1}{\bf u}_{h})\cdot{\bf w}-(r+\nabla\cdot{\bf
w})p_{h}||_{K}^{2}+\beta_{K}^{2}||S^{-1}{\bf u}_{h}||_{K}^{2}.
\end{equation}
Here
 the
residual weight factors
\begin{equation}\label{new 13}
\alpha_{K}:=\min\{\frac{h_{K}}{\sqrt{c_{S,K}}},\frac{1}{\sqrt{c_{{\bf
w},r,K}}}\},\ \ \ \beta_{K}:=C_{{\bf w},r,K}h_{K}\alpha_{K}.
\end{equation}
Note that in
(\ref{new 13}), if $c_{{\bf w},r,K}=0$, $\alpha_{K}$ should be
understood as $h_{K}/\sqrt{c_{S,K}}$.

Let $\nu_{\sigma}$ be given in (\ref{upwind-weighted scheme 5})
for each side $\sigma\in\varepsilon_{h}$. We denote
\begin{equation}\label{hat-hat-p1}
 \hat{\hat{p}}_{\sigma} :=\left \{ \begin{array}{ll}
  (1/2-\nu_{\sigma})(p_{K}-p_{L})\ \  & \mbox{if}\;  \ \ w_{K,\sigma}\geq0,\\
 (1/2-\nu_{\sigma})(p_{L}-p_{K})\ \  & \mbox{if}\; \ \
 w_{K,\sigma}<0
 \end{array}\right.
\end{equation}
when $\sigma$ is an interior side sharing by elements $K$ and $L$,
and
\begin{equation}\label{hat-hat-p2}
 \hat{\hat{p}}_{\sigma} :=\left \{ \begin{array}{ll}
  -\nu_{\sigma}p_{K}\ \  & \mbox{if}\;  \ \ w_{K,\sigma}\geq0,\\
 -(1-\nu_{\sigma})p_{K}\ \  & \mbox{if}\; \ \
 w_{K,\sigma}<0,
 \end{array}\right.
\end{equation}
when $\sigma$ is a boundary side included in $\varepsilon_{K}$.
We thus define an elementwise upwind estimator $\eta_{U,K}$ by
\begin{equation}\label{new 2}
\eta_{U,K}^{2}
:=\frac{h_{K}}{c_{S,K}}\sum\limits_{\sigma\in\varepsilon_{K}} (({\bf
w}\cdot{\bf
n})|_{\sigma})^{2}(||\hat{\hat{p}}_{\sigma}||_{\sigma}^{2}+
h_{\sigma}||S^{-1}{\bf u}_{h}||_{\omega_{\sigma}}^{2}).
\end{equation}

In order to reflect the change of the maximum eigenvalue of the
coefficients matrix $S$ over the patch $\tilde{\omega}_{\sigma}$ of
a side $\sigma\in\varepsilon_{h}$, we introduce a quantity
\begin{equation*}
\Lambda_{\sigma}
:=\max_{K,\bar{K}\cap\bar{\sigma}\neq\emptyset}\{C_{S,K}\}.
\end{equation*}
Similarly, the change of one variation $c_{{\bf w},r,K}$ of the
coefficients over the patch $\tilde{\omega}_{K}$ of an element
$K\in\mathcal{T}_{h}$ is described by the  quantity
\begin{equation*}
\Lambda_{{\bf w},r,K}
:=\max_{K',\bar{K'}\cap\bar{K}\neq\emptyset}\{c_{{\bf w},r,K}\}.
\end{equation*}
Thus we define $\eta_{NC,K}$ as the elementwise nonconforming
estimator by
\begin{equation}\label{new 3}
\eta_{NC,K}^{2} :=\displaystyle\Lambda_{{\bf
w},r,K}h_{K}^{2}||S^{-1}{\bf u}_{h}||_{K}^{2}
+\sum\limits_{\sigma\in\varepsilon_{K}}
\delta_{\sigma}\Lambda_{\sigma}h_{\sigma}||[\gamma_{{\bf
t}_{\sigma}}(S^{-1}{\bf u}_{h})]||_{\sigma}^{2},
\end{equation}
where $\delta_{\sigma}=1/2$ if $\sigma\in\varepsilon_{h}^{\rm int}$,
$\delta_{\sigma}=1$ if $\sigma\in\varepsilon_{h}^{\rm ext}$.

Since the convection occurs in the equations, we need to define two
numbers $\Lambda_{\nabla\cdot{\bf w},K}$ and $\Lambda_{{\bf
w},\sigma}$ similar to P\'{e}clet numbers describing the
convection-dominated. To this end, for each $K\in\mathcal{T}_{h}$ we denote
\begin{equation*}
C_{\nabla\cdot{\bf w},K}
:=|\nabla\cdot{\bf w}|_{K}|, \ \ \Lambda_{\nabla\cdot{\bf w},K}
:=\max_{K':\bar{K'}\cap\bar{K}\neq\emptyset}
\{\frac{C_{\nabla\cdot{\bf w},K'}}{\sqrt{c_{{\bf w},r,K'}}}\},
\end{equation*}
and for    each
$\sigma\in\varepsilon_{h}$  we set $\Lambda_{{\bf w},\sigma} :=\min\{\lambda_{{\bf
w},\sigma},p_{{\bf w},\sigma}\}$ with
\begin{equation*}
 \lambda_{{\bf w},\sigma}
:=\max_{K:\bar{K}\cap\bar{\sigma}\neq\emptyset}\{\frac{C_{{\bf
w},K}}{\sqrt{c_{{\bf w},r,K}}}\},\ \  \ p_{{\bf w},\sigma}
:= \max_{K:\bar{K}\cap\bar{\sigma}\neq\emptyset}\{\frac{h_{K}C_{{\bf
w},K}}{\sqrt{c_{S,K}}}\}.
\end{equation*}
We then define $\eta_{C,K}$ as an elementwise convection
estimator by
\begin{equation}\label{new 4}
\eta_{C,K}^{2} :=\displaystyle \Lambda_{\nabla\cdot{\bf
w},K}^{2}h_{K}^{2}||S^{-1}{\bf u}_{h}||_{K}^{2}
+\sum\limits_{\sigma\in\varepsilon_{K}}\delta_{\sigma}\Lambda_{{\bf
w},\sigma}^{2}h_{\sigma}||[\gamma_{{\bf t}_{\sigma}}(S^{-1}{\bf
u}_{h})]||_{\sigma}^{2}.
\end{equation}

We now state  a posteriori error estimates for the global error of stress and
displacement.
\begin{theorem} \label{global error estimate} {\rm (Global
error estimate for the centered mixed scheme)}\ {\it Let $p\in
H_{0}^{1}(\Omega)$ be the weak solution of the problem
(\ref{continuous problem 2}), ${\bf u}=-S\nabla p$ be the continuous
stress vector, $({\bf u}_{h},p_{h})$ be the solution of the centered
mixed scheme (\ref{centered mixed scheme})-(\ref{centered mixed
scheme 2}).  Let $\mathcal{E}$ be the error of  the stress and
displacement in the   weighted norm defined in (\ref{new-1}),
$\eta_{D,K},\eta_{R,K},\eta_{NC,K}$, and $\eta_{C,K}$ are the
corresponding elementwise displacement estimator, residual
estimator, convection estimator, and nonconforming estimator,
defined in (\ref{new 0})-(\ref{new 1}) and (\ref{new 3})-(\ref{new
4}), respectively. Then it holds
\begin{equation}\label{global error estimate 1}
\mathcal{E}\leq
c_{1}\{\sum\limits_{K\in\mathcal{T}_{h}}(\eta_{D,K}^{2}+\eta_{R,K}^{2}+\eta_{C,K}^{2}
+\eta_{NC,K}^{2})\}^{1/2}.
\end{equation}}
\end{theorem}
\begin{theorem} \label{Global upwind} {\rm (Global
error estimate for the upwind-weighted scheme)}\ {\it Let $p\in
H_{0}^{1}(\Omega)$ be the weak solution of the problem
(\ref{continuous problem 2}), ${\bf u}=-S\nabla p$ be the continuous
stress vector, $({\bf u}_{h},p_{h})$ be the solution of the
upwind-weighted mixed scheme (\ref{upwind-weighted scheme
1})-(\ref{upwind-weighted scheme 2}).  Let   $\mathcal{E}$ be the
error of  the stress and displacement in the   weighted norm defined
in (\ref{new-1}), $\eta_{D,K},\eta_{R,K},\eta_{U,K}$, $\eta_{NC,K}$,
and $\eta_{C,K}$ are the corresponding elementwise displacement
estimator, residual estimator, upwind estimator, convection
estimator, and nonconforming estimator, defined in (\ref{new
0})-(\ref{new 1}) and (\ref{new 2})-(\ref{new 4}), respectively.
Then it holds
\begin{equation}\label{global error estimate 2}
\mathcal{E}\leq
c_{2}\{\sum\limits_{K\in\mathcal{T}_{h}}(\eta_{D,K}^{2}+\eta_{R,K}^{2}+\eta_{C,K}^{2}
+\eta_{NC,K}^{2}+\eta_{U,K}^{2})\}^{1/2}.
\end{equation}}
\end{theorem}

\begin{remark}\label{robust Remark}
\ {\it We note that the constants $c_{1}$ in (\ref{global error estimate
1}) and $c_{2}$ in (\ref{global error estimate 2}) only depend on the spatial dimension and the shape regularity parameter of the triangulation $\mathcal{T}_{h}$, and are
independent of the coefficients $S, {\bf w}, r$. In this sense, the proposed estimators are robust  with respect to all the coefficients.}
\end{remark}

\begin{remark}\ {\it In \cite{Carstensen0} Carstensen presented a posteriori error
estimates of the
Raviart-Thomas,
  Brezzi-Douglas-Morini, Brezzi-Douglas-Fortin-Marini elements $(M_{h},L_{h})$
  for the diffusion equations (the case ${\bf w}=r=0$ in the model (\ref{convection-diffufsion-equations1})).
 In his estimators,  the term
$\displaystyle\min_{v_{h}\in L_{h}}||h(S^{-1}{\bf
u}_{h}-\nabla_{h}v_{h})||_{\Omega}$ is included.  In practice one
may substitute it with the term
$||h(S^{-1}{\bf u}_{h}-\nabla_{h}p_{h})||_{\Omega}$, where
 $({\bf u}_{h},p_{h})\in M_{h}\times L_{h}$ is a pair
 of finite element solutions.  For the lowest order
Raviart-Thomas element, it holds $\nabla_{h}p_{h}=0$, then $||h(S^{-1}{\bf
u}_{h}-\nabla_{h}p_{h})||_{\Omega}$ is reduced to $||hS^{-1}{\bf
u}_{h}||_{\Omega}$, which shows that occurrence of $||hS^{-1}{\bf
u}_{h}||_{\Omega}$ is reasonable in the a posteriori error
estimators $\eta_{D,K}$ defined in  (\ref{new 0}).
In addition,  we note that  the postprocessing (\ref{postprocessed variable 1}) can remove the term $||h
{\rm curl}(S^{-1}{\bf u}_{h})||_{\Omega}$, which is also contained in Carstensen's estimators.}
\end{remark}

The global error estimates above show that the a posteriori
indicator over each element consists of a series of estimators.
Thus, the local efficiency of each component ensures the local
efficiency of the a posteriori indicator over an element. Here, we
point out the local efficiency is in the sense that its converse
estimate holds up to a different multiplicative constant.


\begin{theorem} \label{new 5}{\rm (Local efficiency for the displacement and
residual estimators)}\ {\it For $K\in\mathcal{T}_{h}$, let $\eta_{D,K}$ and $\eta_{R,K}$
denote the elementwise displacement and residual estimators defined
in (\ref{new 0}) and (\ref{new 1}), respectively. Then it holds
\begin{equation}\label{new 6}
(\eta_{D,K}^{2}+\eta_{R,K}^{2})^{1/2}\leq
c_{3}\alpha_{*,K}\mathcal{E}_{K}
\end{equation}}
with \begin{equation*}
\begin{array}{lll}
\alpha_{*,K}
:&=&\max\{\sqrt{\frac{C_{S,K}}{c_{S,K}}}+\frac{h_{K}C_{{\bf
w},K}}{c_{S,K}},\frac{h_{K}C_{{\bf w},r,K}}{\sqrt{c_{{\bf
w},r,K}c_{S,K}}}\}\vspace{2mm}\\
& &+\max\{\frac{h_{K}^{2}C_{{\bf
w},r,K}}{c_{S,K}},\frac{h_{K}C_{{\bf w},r,K}}{\sqrt{c_{S,K}c_{{\bf
w},r,K}}}\}+\max\{\frac{h_{K}\sqrt{c_{{\bf
w},r,K}}}{\sqrt{c_{S,K}}},1\}.
\end{array}
\end{equation*}

\end{theorem}


\begin{theorem}\label{new 7}{\rm (Local efficiency for the nonconforming and
convection estimators)}\ {\it Let $\eta_{NC,K}$ and $\eta_{C,K}$ be
the elementwise nonconforming and convection estimators defined in
(\ref{new 3}) and (\ref{new 4}), respectively.
Then it holds
\begin{equation}\label{new 8}
\{\eta_{NC,K}^{2}+\eta_{C,K}^{2}\}^{1/2}\leq
c_{4}\{\beta_{*,K}^{2}\mathcal{E}_{K}^{2}+
\sum\limits_{\sigma\in\varepsilon_{K}}
c_{\omega_{\sigma}}^{2}(\Lambda_{\sigma}+\Lambda_{{\bf
w},\sigma}^{2})||S^{-1/2}({\bf u}-{\bf
u}_{h})||_{\omega_{\sigma}}^{2}\}^{1/2},
\end{equation}
}
where 
\begin{equation*}
\beta_{*,K}^{2} :=(\Lambda_{{\bf w},r,K}+\Lambda_{\nabla\cdot{\bf
w},K}^{2})\max\{h_{K}^{2}/c_{S,K},1/c_{{\bf w},r,K}\},
\end{equation*}
\begin{equation*}
c_{\omega_{\sigma}} :=\left \{ \begin{array}{ll}
  \max(c_{S,K}^{-1/2},c_{S,L}^{-1/2})\ \  & \mbox{if}\;  \ \ \sigma=\bar{K}\cap \bar{L},\\
 c_{S,K}^{-1/2}\ \ \ \  & \mbox{if}\; \ \
 \sigma\in\varepsilon_{K}\cap\varepsilon_{h}^{\rm ext},
 \end{array}\right.
\end{equation*}
and $\Lambda_{{\bf w},r,K}, \Lambda_{\nabla\cdot{\bf
w},K}$ are the same as in (\ref{new 3})-(\ref{new 4}).
\end{theorem}

 We finally need the following quantities for the local efficiency of the upwind estimator over
an element, where $\nu_{\sigma}$ is  given in (\ref{upwind-weighted scheme 5})
for each side $\sigma\in\varepsilon_{h}$. 
\begin{equation*}
\lambda_{\sigma} :=\left \{ \begin{array}{ll}\frac{|({\bf
w}\cdot{\bf
n})|_{\sigma}|}{\sqrt{c_{S,K}}}\left((\frac{1}{2}-\nu_{\sigma})\max(\frac{1}{\sqrt{c_{S,K}}},\frac{1}{\sqrt{c_{S,L}}})+
\max(\frac{h_{K}}{\sqrt{c_{S,K}}},\frac{h_{L}}{\sqrt{c_{S,L}}})\right)\ & \mbox{if}\; \ \sigma=\bar{K}\cap \bar{L},\vspace{2mm}\\
\frac{|({\bf w}\cdot{\bf
n})|_{\sigma}|}{\sqrt{c_{S,K}}}\left((1-\nu_{\sigma})\frac{1}{\sqrt{c_{S,K}}}+
\frac{h_{K}}{\sqrt{c_{S,K}}}\right)\ & \mbox{if}\;\ \
\sigma\in\varepsilon_{K}^{\rm ext},
 \end{array}\right.
\end{equation*}
\begin{equation*}
\rho_{\sigma} :=\left \{ \begin{array}{ll}\frac{|({\bf w}\cdot{\bf
n})|_{\sigma}|}{\sqrt{c_{S,K}}}\left((\frac{1}{2}-\nu_{\sigma})|\sigma|^{-\frac{1}{2}}+1\right)\max(\frac{1}{\sqrt{c_{{\bf
w},r,K}}},\frac{1}{\sqrt{c_{{\bf w},r,L}}})\ \ & \mbox{if}\;\ \ \sigma=\bar{K}\cap \bar{L},\vspace{2mm}\\
\frac{|({\bf w}\cdot{\bf
n})|_{\sigma}|}{\sqrt{c_{S,K}}}\left((1-\nu_{\sigma})|\sigma|^{-\frac{1}{2}}+1\right)\frac{1}{\sqrt{c_{{\bf
w},r,K}}} \ \ & \mbox{if}\;\ \sigma\in\varepsilon_{K}^{\rm ext},
 \end{array}\right.
\end{equation*}
and
\begin{equation*}
\mathcal{E}_{D,\omega_{\sigma}} :=\left \{ \begin{array}{ll}
\left(c_{{\bf w},r,K}||p-p_{h}||_{K}^{2}+c_{{\bf
w},r,L}||p-p_{h}||_{L}^{2}\right)^{1/2} \ \ & \mbox{if}\; \ \
\sigma=\bar{K}\cap \bar{L},\vspace{2mm}\\
\sqrt{c_{{\bf w},r,K}}||p-p_{h}||_{K}\ \ & \mbox{if}\;\ \
\sigma\in\varepsilon_{K}^{\rm ext}.
\end{array}\right.
\end{equation*}
\begin{theorem}\label{new 9}{\rm (Local efficiency for the upwind estimator)}\ {\it
Let  $\eta_{U,K}$ be the elementwise upwind estimator defined in
(\ref{new 2}). Then, it holds
\begin{equation}\label{new 10}
\eta_{U,K}\leq
c_{5}\sum\limits_{\sigma\in\varepsilon_{K}}\left(\lambda_{\sigma}
||S^{-1/2}({\bf u}-{\bf
u}_{h})||_{\omega_{\sigma}}+\rho_{\sigma}\mathcal{E}_{D,\omega_{\sigma}}\right).
\end{equation}}
\end{theorem}

\section{Preliminary results and remarks}
In this section, firstly we  show   the abstract error estimates
developed by {\it Vohral\'{\i}k} in \cite{Vohralik1}, and then make
some remarks on {\it Vohral\'{\i}k}'s a posteriori error estimators.
To this end, for any $\varphi\in H_{0}^{1}(\Omega)$ we define
\begin{equation}\label{residual component 1}
T_{R}(\varphi)
:=\sum\limits_{K\in\mathcal{T}_{h}}(f+\nabla\cdot(S\nabla
\tilde{p}_{h})-\nabla\cdot(\tilde{p}_{h}{\bf
w})-r\tilde{p}_{h},\varphi-\varphi_{K}),
\end{equation}
\begin{equation}\label{covection component 1}
T_{C}(\varphi,s)
:=\sum\limits_{K\in\mathcal{T}_{h}}(\nabla\cdot((\tilde{p}_{h}-s){\bf
w})-1/2(\tilde{p}_{h}-s)\nabla\cdot{\bf w},\varphi)_{K},
\end{equation}
\begin{equation}\label{upwind component 1}
T_{U}(\varphi) :=\sum\limits_{K\in\mathcal{T}_{h}}
\sum\limits_{\sigma\in\varepsilon_{K}}<(\hat{p}_{\sigma}-\tilde{p}_{h}){\bf
w}\cdot{\bf n},\varphi_{K}>_{\sigma},
\end{equation}
where $\varphi_{K}$  is the
mean of $\varphi$ over $K$, $s\in H_{0}^{1}(\Omega)$ is arbitrarily given, $\tilde{p}_{h}$ is the
postprocessed approximation solution given by
(\ref{postprocessed variable 1})-(\ref{postprocessed variable
2}), and $\hat{p}_{\sigma}$ is the weighted upwind value defined in
(\ref{upwind-weighted scheme 3})-(\ref{upwind-weighted scheme
4}).
\begin{lemma}\label{abstract error estimate}
{\rm (Abstract error estimates by Vohral\'{\i}k)}\ Let $p\in
H_{0}^{1}(\Omega)$ be the weak solution of the problem
(\ref{continuous problem 2}), and let $s\in H_{0}^{1}(\Omega)$ be
arbitrary.
Then it holds
\begin{equation}\label{abstract error estimates-centered}
|||p-\tilde{p}_{h}|||_{\Omega}\leq|||\tilde{p}_{h}-s|||_{\Omega}+
\sup_{\varphi\in
H_{0}^{1}(\Omega),|||\varphi|||_{\Omega}=1}\{T_{R}(\varphi)+T_{C}(\varphi,s)\}
\end{equation}
if $\tilde{p}_{h}$ is the postprocessed solution, given by
(\ref{postprocessed variable 1})-(\ref{postprocessed variable 2}),
of the centered mixed finite element scheme (\ref{centered mixed
scheme})-(\ref{centered mixed scheme 2}),  and holds
\begin{equation}\label{abstract error estimates-upwind}
|||p-\tilde{p}_{h}|||_{\Omega}\leq|||\tilde{p}_{h}-s|||_{\Omega}+
\sup_{\varphi\in
H_{0}^{1}(\Omega),|||\varphi|||_{\Omega}=1}\{T_{R}(\varphi)+T_{C}(\varphi,s)+T_{U}(\varphi)\}
\end{equation}
if $\tilde{p}_{h}$ is the postprocessed solution, given by
(\ref{postprocessed variable 1})-(\ref{postprocessed variable 2}), of the
upwind-weighted mixed finite element scheme (\ref{upwind-weighted
scheme 1})-(\ref{upwind-weighted scheme 2}).
\end{lemma}
\begin{remark}\
\  In Vohral\'{\i}k's work \cite{Vohralik1},  the modified
Oswald interpolation, $\mathcal{I}_{\rm MO}(\tilde{p}_{h})\in
H_{0}^{1}(\Omega)$, of $\tilde{p}_{h}$ is introduced to replace $s$
in the abstract error estimates (\ref{abstract error
estimates-centered})-(\ref{abstract error estimates-upwind}) so as
to obtain  computable estimates of the terms.
\end{remark}

We now state our abstract error estimates for the global error of
stress and displacement  in the weighted norm.
\begin{lemma}\label{abstract error estimate 1}
({\rm Abstract error estimates for the global error})\
{\it Let $p\in H_{0}^{1}(\Omega)$ denote the weak
solution of the problem (\ref{continuous problem 2}), and  $s\in
H_{0}^{1}(\Omega)$ be arbitrary. Let   $\mathcal{E}$ be the global error defined in
(\ref{new-1}) and  $\eta_{D,K}$ be the elementwise displacement estimator
defined in (\ref{new 0}). Then it holds
\begin{equation}\label{abstract error estimate-centered 1}
\mathcal{E}\leq\sqrt{2}\{|||\tilde{p}_{h}-s|||_{\Omega}+
\sup_{\varphi\in
H_{0}^{1}(\Omega),|||\varphi|||_{\Omega}=1}(T_{R}(\varphi)+
T_{C}(\varphi,s))+(\sum\limits_{K\in\mathcal{T}_{h}}\eta_{D,K}^{2})^{1/2}\}
\end{equation}
if $\tilde{p}_{h}$ is the postprocessed solution, given by
(\ref{postprocessed variable 1})-(\ref{postprocessed variable 2}),
of the centered mixed finite (\ref{centered mixed
scheme})-(\ref{centered mixed scheme 2}), and holds
\begin{equation}\label{abstract error estimate-upwind 1}
\begin{array}{lll}
\mathcal{E}\leq\sqrt{2}\{|||\tilde{p}_{h}-s|||_{\Omega}&+&
\sup_{\varphi\in
H_{0}^{1}(\Omega),|||\varphi|||_{\Omega}=1}(T_{R}(\varphi)+T_{C}(\varphi,s)+T_{U}(\varphi))\vspace{2mm}\\
&+&\displaystyle(\sum\limits_{K\in\mathcal{T}_{h}}\eta_{D,K}^{2})^{1/2}\}
\end{array}
\end{equation}}
if  $\tilde{p}_{h}$ is the postprocessed solution, given by
(\ref{postprocessed variable 1})-(\ref{postprocessed variable 2}), of the
upwind-weighted mixed finite element scheme (\ref{upwind-weighted
scheme 1})-(\ref{upwind-weighted scheme 2}).
\end{lemma}
\begin{proof}\
By the postprocessed formulations (\ref{postprocessed variable 1})-(\ref{postprocessed variable 2}) and
the generalized Friedrichs
inequality (\ref{Friedrichs inequality}), we have
\begin{equation}\label{idea 1}
\begin{array}{lll}
||p-p_{h}||_{K}&\leq&||p-\tilde{p}_{h}||_{K}+||\tilde{p}_{h}-p_{h}||_{K}
\leq||p-\tilde{p}_{h}||_{K}+h_{K}||\nabla\tilde{p}_{h}||_{K}\vspace{2mm}\\
&=&||p-\tilde{p}_{h}||_{K}+h_{K}||S^{-1}{\bf u}_{h}||_{K}\ \ \ {\rm
for\ all}\ \  K\in\mathcal{T}_{h}.
\end{array}
\end{equation}
On the other hand, it holds
\begin{equation}\label{idea 2}
||S^{-1/2}({\bf u}-{\bf
u}_{h})||_{K}^{2}=||S^{1/2}\nabla(p-\tilde{p}_{h})||_{K}^{2}\ \ \
{\rm  for\ all}\ \ \ K\in\mathcal{T}_{h}.
\end{equation}
Summing (\ref{idea 2}) and  (\ref{idea 1}) with  a multiplier $c_{{\bf
w},r,K}^{1/2}$   over all $K\in\mathcal{T}_{h}$ yields
\begin{equation}\label{idea 3}
\mathcal{E}\leq \sqrt{2}(|||p-\tilde{p}_{h}|||_{\Omega}+
\{\sum\limits_{K\in\mathcal{T}_{h}}c_{{\bf
w},r,K}h_{K}^{2}||S^{-1}{\bf u}_{h}||_{K}^{2}\}^{1/2}).
\end{equation}
 The desired results (\ref{abstract error estimate-centered 1})-(\ref{abstract error estimate-upwind 1}) then follows from   LEMMA \ref{abstract error estimate}.
\end{proof}
\begin{lemma}\label{Bramble-Hilbert lemma}
\ For any $K\in\mathcal{T}_{h}$ and  $\varphi\in H^{1}(K)$, it holds
\begin{equation}\label{Bramble-hilbert lemma 1}
||\varphi-\varphi_{K}||_{K}\leq c_{6}\alpha_{K}|||\varphi|||_{K},
\end{equation}
where $\varphi_{K}$ denotes the mean of $\varphi$ over $K$, and
$\alpha_{K}$ is defined as in (\ref{new 13}).
\end{lemma}
\begin{proof} From  (\ref{new
13}), it holds $\alpha_{K}=h_{K}c_{S,K}^{-1/2}$ when $h_{K}c_{S,K}^{-1/2}\leq c_{{\bf w},r,K}^{-1/2}$. By Bramble-Hilbert lemma
we have
\begin{equation}
\begin{array}{lll}\label{Bramble-Holbert lemma 2}
||\varphi-\varphi_{K}||_{K}&\leq&
c_{7}h_{K}||\nabla\varphi||_{K}\leq
c_{7}h_{K}c_{S,K}^{-1/2}||S^{1/2}\nabla\varphi||_{K}\vspace{2mm}\\
&=&c_{7}\alpha_{K}||S^{1/2}\nabla\varphi||_{K}\leq
c_{7}\alpha_{K}|||\varphi|||_{K}.
\end{array}
\end{equation}
On the other hand, when  $h_{K}c_{S,K}^{-1/2}>c_{{\bf w},r,K}^{-1/2}$, it holds $\alpha_{K}=c_{{\bf w},r,K}^{-1/2}$. By the property of
$L^{2}-$projection we get
\begin{equation}\label{Bramble-Hilbert lemma 3}
\begin{array}{lll}
||\varphi-\varphi_{K}||_{K}&\leq&||\varphi||_{K}=c_{{\bf
w},r,K}^{-1/2}c_{{\bf w},r,K}^{1/2}||\varphi||_{K}\vspace{2mm}\\
&=&\alpha_{K}c_{{\bf w},r,K}^{1/2}||\varphi||_{K}\leq
\alpha_{K}|||\varphi|||_{K}.
\end{array}
\end{equation}
 The assertion
(\ref{Bramble-hilbert lemma 1}) follows from (\ref{Bramble-Holbert
lemma 2})-(\ref{Bramble-Hilbert lemma 3}) with $c_{6} :=\max\{c_{7},
1\}$.
\end{proof}

\section{A posteriori error analysis}\ We devote this section
to   computable estimates of $T_{R}(\varphi),T_{U}(\varphi)$ and
$T_{C}(\varphi,s)$ defined in (\ref{residual component 1}),
(\ref{upwind component 1}) and (\ref{covection component 1}),
respectively,  with the help of ${\bf u}_{h}$ and $p_{h}$. Moreover,
we derive an estimate of $|||\tilde{p}_{h}-s|||$ by substituting $s$
with the modified Oswald interpolation $I_{MO}(\tilde{p}_{h})$ (see
\cite{Vohralik1}), and by using the postprocessing technique as a
transition.  Finally, we give the proof of THEOREMs \ref{global error
estimate}-\ref{Global upwind}.

\begin{lemma}\label{residual estimator (global)}
{\rm (Residual estimator)}\ Let $T_{R}(\varphi)$ be  defined  as in
(\ref{residual component 1}) with $|||\varphi|||_{\Omega}=1$, and
$\eta_{R,K}$ be  defined as in (\ref{new 1}). Then it holds
\begin{equation}\label{residual estimator 1}
T_{R}(\varphi)\leq
c_{8}\{\sum\limits_{K\in\mathcal{T}_{h}}\eta_{R,K}^{2}\}^{1/2}.
\end{equation}
\end{lemma}
\begin{proof}
A
combination of Assumption (D4), LEMMA \ref{Bramble-Hilbert
lemma},   Friedrichs inequality (\ref{Friedrichs inequality}),
and the postprocessing (\ref{postprocessed variable 1})-(\ref{postprocessed variable 2}), yields
\begin{equation}\label{residual estimator 7}
\begin{array}{lll}
T_{R}(\varphi)&=&\displaystyle\sum\limits_{K\in\mathcal{T}_{h}}(f+
\nabla\cdot(S\nabla\tilde{p}_{h})-\nabla\cdot(\tilde{p}_{h}{\bf
w})-r\tilde{p}_{h},\varphi-\varphi_{K})_{K}\vspace{2mm}\\
&=&\displaystyle\sum\limits_{K\in\mathcal{T}_{h}}(f-\nabla\cdot{\bf
u}_{h}+(S^{-1}{\bf u}_{h})\cdot{\bf w}-(r+\nabla\cdot{\bf
w})p_{h},\varphi-\varphi_{K})_{K}\vspace{2mm}\\
&\
&\displaystyle+\sum\limits_{K\in\mathcal{T}_{h}}((r+\nabla\cdot{\bf
w})(p_{h}-\tilde{p}_{h}),\varphi-\varphi_{K})_{K}\vspace{2mm}\\
&\leq&\displaystyle c_{8}\{\sum\limits_{K\in\mathcal{T}_{h}}
\alpha_{K}||f-\nabla\cdot{\bf u}_{h}+(S^{-1}{\bf u}_{h})\cdot{\bf
w}-(r+\nabla\cdot{\bf
w})p_{h}||_{K}|||\varphi|||_{K}\vspace{2mm}\\
&\ &\displaystyle+\sum\limits_{K\in\mathcal{T}_{h}}C_{{\bf
w},r,K}h_{K}||\nabla\tilde{p}_{h}||_{K}\alpha_{K}|||\varphi|||_{K}\}\vspace{2mm}\\
&\leq&\displaystyle c_{8}\{\sum\limits_{K\in\mathcal{T}_{h}}
\alpha_{K}||f-\nabla\cdot{\bf u}_{h}+(S^{-1}{\bf u}_{h})\cdot{\bf
w}-(r+\nabla\cdot{\bf
w})p_{h}||_{K}|||\varphi|||_{K}\vspace{2mm}\\
&\ &\displaystyle+\sum\limits_{K\in\mathcal{T}_{h}}
\beta_{K}||S^{-1}{\bf u}_{h}||_{K}|||\varphi|||_{K}\}.
\end{array}
\end{equation}
Then the desired result (\ref{residual estimator 1}) follows with $|||\varphi|||_{\Omega}=1$.
\end{proof}

\begin{lemma}\label{upwind estimator (global)}
{\rm(Upwind estimator)}\ Let $T_{U}(\varphi)$ be defined as in
(\ref{upwind component 1}) with $|||\varphi|||_{\Omega}=1$, and
$\eta_{U,K}$ be   defined as in (\ref{new 2}). Then it holds
\begin{equation}\label{upwind estimator 1}
T_{U}(\varphi)\leq
c_{9}\{\sum\limits_{K\in\mathcal{T}_{h}}\eta_{U,K}^{2}\}^{1/2}.
\end{equation}
\end{lemma}
\begin{proof} We denote by $\tilde{p}_{\sigma}$ the mean of
$\tilde{p}_{h}$ over $\sigma\in\varepsilon_{h}$, i.e.,
$\tilde{p}_{\sigma} :=<1,\tilde{p}_{h}>_{\sigma}/|\sigma|$. The
definitions of $T_{U}(\varphi)$ and $w_{K,\sigma}$, together with
Assumption $(D2)$ of the velocity field ${\bf w}$, imply
\begin{equation}\label{upwind estimator 2}
T_{U}(\varphi)=\sum\limits_{K\in\mathcal{T}_{h}}\sum\limits_{\sigma\in\varepsilon_{K}}
(\hat{p}_{\sigma}-\tilde{p}_{\sigma})w_{K,\sigma}\varphi_{K}.
\end{equation}

For an element $K\in\mathcal{T}_{h}$, it holds $\sigma\in\varepsilon_{K}\cap\varepsilon_{L}$ or
$\sigma\in\varepsilon_{K}^{{\rm ext}}$. For the former
case, recalling $p_{K}=p_{h}|_{K},p_{L}=p_{h}|_{L}$,  from the
postprocessing (\ref{postprocessed variable 2}) we obtain
\begin{equation}\label{upwind estimator 3}
\begin{array}{lll}
\hat{p}_{\sigma}-\tilde{p}_{\sigma}&=&\displaystyle\hat{p}_{\sigma}-
\frac{1}{2}(p_{K}+p_{L})+\frac{1}{2}(p_{K}-\tilde{p}_{\sigma})+
\frac{1}{2}(p_{L}-\tilde{p}_{\sigma})\vspace{2mm}\\
&=&\displaystyle\hat{p}_{\sigma}-
\frac{1}{2}(p_{K}+p_{L})+\frac{1}{2}(\frac{1}{|K|}\int_{K}\tilde{p}_{h}dx
-\frac{1}{|\sigma|}\int_{\sigma}\tilde{p}_{h}ds)\vspace{2mm}\\
&\ &\displaystyle+\frac{1}{2}(\frac{1}{|L|}\int_{L}\tilde{p}_{h}dx
-\frac{1}{|\sigma|}\int_{\sigma}\tilde{p}_{h}ds).
\end{array}
\end{equation}
For the latter case,  we similarly have
\begin{equation}\label{upwind estimator 4}
\hat{p}_{\sigma}-\tilde{p}_{\sigma}=\hat{p}_{\sigma}-p_{K}+
(\frac{1}{|K|}\int_{K}\tilde{p}_{h}dx
-\frac{1}{|\sigma|}\int_{\sigma}\tilde{p}_{h}ds).
\end{equation}
For convenience, in what follows we denote
\begin{equation*}
 \hat{p}_{\omega_{\sigma}} :=\frac{1}{2}(\frac{1}{|K|}\int_{K}\tilde{p}_{h}dx
-\frac{1}{|\sigma|}\int_{\sigma}\tilde{p}_{h}ds)+
\frac{1}{2}(\frac{1}{|L|}\int_{L}\tilde{p}_{h}dx
-\frac{1}{|\sigma|}\int_{\sigma}\tilde{p}_{h}ds)
\end{equation*}
when $\sigma\in\varepsilon_{K}\cap\varepsilon_{L}$, and
\begin{equation*}
 \hat{p}_{\omega_{\sigma}} :=
 \frac{1}{|K|}\int_{K}\tilde{p}_{h}dx
-\frac{1}{|\sigma|}\int_{\sigma}\tilde{p}_{h}ds
\end{equation*}
when $\sigma\in\varepsilon_{K}^{{\rm ext}}$.

In light of the definitions of $\hat{p}_{\sigma}$ and
$\hat{\hat{p}}_{\sigma}$  in (\ref{upwind-weighted scheme
3})-(\ref{upwind-weighted scheme 4})  and
(\ref{hat-hat-p1})-(\ref{hat-hat-p2}), and from (\ref{upwind
estimator 2})-(\ref{upwind estimator 4}) we have
\begin{equation}\label{upwind estimator 5}
T_{U}(\varphi)=\sum\limits_{K\in\mathcal{T}_{h}}
\sum\limits_{\sigma\in\varepsilon_{K}}(\hat{\hat{p}}_{\sigma}+
\hat{p}_{\omega_{\sigma}})w_{K,\sigma}\varphi_{K}.
\end{equation}
Since $\varphi\in
H_{0}^{1}(\Omega)$, and $\hat{\hat{p}}_{\sigma}$,
$\hat{p}_{\omega_{\sigma}}$ are   constants over
a side $\sigma\in\varepsilon_{h}$,  it holds
\begin{equation}\label{upwind estimator 6}
\sum\limits_{K\in\mathcal{T}_{h}}
\sum\limits_{\sigma\in\varepsilon_{K}}\hat{\hat{p}}_{\sigma}w_{K,\sigma}\varphi_{K}=
\sum\limits_{K\in\mathcal{T}_{h}}
\sum\limits_{\sigma\in\varepsilon_{K}}\int_{\sigma}\hat{\hat{p}}_{\sigma}{\bf
w}\cdot{\bf n}(\varphi_{K}-\varphi),
\end{equation}
\begin{equation}\label{upwind estimator 7}
\sum\limits_{K\in\mathcal{T}_{h}}
\sum\limits_{\sigma\in\varepsilon_{K}}\hat{p}_{\omega_{\sigma}}w_{K,\sigma}\varphi_{K}=
\sum\limits_{K\in\mathcal{T}_{h}}
\sum\limits_{\sigma\in\varepsilon_{K}}\int_{\sigma}\hat{p}_{\omega_{\sigma}}{\bf
w}\cdot{\bf n}(\varphi_{K}-\varphi).
\end{equation}
From Friedrichs inequality (\ref{Friedrichs inequality}) and the
postprocessing (\ref{postprocessed variable 1}) we have
\begin{equation}\label{upwind estimator 8}
|\hat{p}_{\omega_{\sigma}}|\leq c_{10}h_{\sigma}^{1-d/2}||S^{-1}{\bf
u}_{h}||_{\omega_{\sigma}}.
\end{equation}
The trace inequality (see LEMMA 3.1 in \cite{Verfurth-4}) and local
shape regularity of elements indicate
\begin{equation}\label{proof1}
\begin{array}{lll}
||\varphi_{K}-\varphi||_{\sigma}&\leq&c_{10}
(h_{\sigma}^{-1/2}||\varphi-\varphi_{K}||_{K}+
||\varphi-\varphi_{K}||_{K}^{1/2}||\nabla(\varphi-\varphi_{K})||_{K}^{1/2})\vspace{2mm}\\
&\leq&c_{11}h_{K}^{1/2}||\nabla\varphi||_{K}\leq c_{11}
h_{K}^{1/2}c_{S,K}^{-1/2}||S^{1/2}\nabla\varphi||_{K}.
\end{array}
\end{equation}
A combination of (\ref{upwind estimator 8})- (\ref{proof1}) then yields
\begin{equation}\label{upwind estimator 10}
\begin{array}{lll}
&\ &\displaystyle\sum\limits_{K\in\mathcal{T}_{h}}
\sum\limits_{\sigma\in\varepsilon_{K}}\int_{\sigma}\hat{p}_{\omega_{\sigma}}{\bf
w}\cdot{\bf n}(\varphi_{K}-\varphi)\vspace{2mm}\\
 &\ &
\leq\ \  c_{11}\displaystyle\sum\limits_{K\in\mathcal{T}_{h}}
\{\sum\limits_{\sigma\in\varepsilon_{K}}|({\bf w}\cdot{\bf
n})|_{\sigma}|h_{\sigma}^{1/2}||S^{-1}{\bf
u}_{h}||_{\omega_{\sigma}}\}h_{K}^{1/2}c_{S,K}^{-1/2}|||\varphi|||_{K}.
\end{array}
\end{equation}
Similarly we can obtain
\begin{equation}\label{upwind estimator 9}
\sum\limits_{K\in\mathcal{T}_{h}}
\sum\limits_{\sigma\in\varepsilon_{K}}\int_{\sigma}\hat{\hat{p}}_{\sigma}{\bf
w}\cdot{\bf n}(\varphi_{K}-\varphi)\leq
c_{12}\sum\limits_{K\in\mathcal{T}_{h}}
\{\sum\limits_{\sigma\in\varepsilon_{K}}|({\bf w}\cdot{\bf
n})|_{\sigma}|||\hat{\hat{p}}_{\sigma}||_{\sigma}\}h_{K}^{1/2}c_{S,K}^{-1/2}|||\varphi|||_{K}.
\end{equation}

Finally, the desired result (\ref{upwind estimator 1}) follows from
(\ref{upwind estimator 5})-(\ref{upwind estimator 7}) and
(\ref{upwind estimator 10})-(\ref{upwind estimator 9}) with
$c_{9}:=\max(c_{11},c_{12})$ and $|||\varphi|||_{\Omega}=1$.
\end{proof}

For the first term, $|||\tilde{p}_{h}-s|||_{\Omega}$, in the right
side of the abstract error estimate (\ref{abstract error
estimate-centered 1}) or (\ref{abstract error estimate-upwind 1}),
we follow \cite{Vohralik1} to take $s:=I_{MO}(\tilde{p}_{h})$ in the
sequel, where  $I_{MO}(\tilde{p}_{h})$ is   the modified Oswald
interpolation of $\tilde{p}_{h}$. Recall an estimate on the modified
Oswald interpolation \cite{Karakashian},
\begin{equation}\label{Oswald interpolation 1}
||\nabla(\varphi_{h}-\mathcal{I}_{{\rm
MO}}(\varphi_{h}))||_{K}^{2}\leq
c_{13}\sum\limits_{\sigma:\sigma\cap
K\neq\Phi}h_{\sigma}^{-1}||[\varphi_{h}]||_{\sigma}^{2},\ \
\varphi_{h}\in\mathbb{P}_{d}(\mathcal{T}_{h})\cap
W_{0}(\mathcal{T}_{h}),
\end{equation}
where $\mathcal{I}_{{\rm
MO}}(\varphi_{h})\in\mathbb{P}_{d}(\mathcal{T}_{h})\cap
H_{0}^{1}(\Omega)$ is the modified Oswald interpolation of
$\varphi_{h}$, $\mathbb{P}_{d}(\mathcal{T}_{h})$ ($d=$2 or 3 )
denotes the set of polynomials of degree at most $d$ on each
simplex,  $\sigma\cap K\neq\emptyset$ when $\sigma$ contains a
vertex of $K$.

By definition we have
$$|||\tilde{p}_{h}-s|||_{\Omega}=\left\{\sum\limits_{K\in\mathcal{T}_{h}}
(S\nabla(\tilde{p}_{h}-s),\nabla(\tilde{p}_{h}-s))_{K}+\sum\limits_{K\in\mathcal{T}_{h}}c_{{\bf w},r,K}||\tilde{p}_{h}-s||_{K}^{2}\right\}^{1/2}.$$ LEMMAs \ref{auxialiary problem 2}-\ref{dual problem 2}       show  respectively computable estimates of the two right-side terms of the above identity with the help
of ${\bf u}_{h}$ and $p_{h}$.
\begin{lemma}\label{auxialiary problem 2}
 Let $\gamma_{{\bf
t}_{\sigma}}(\cdot)$ be defined as in Section 2.1, and
$s:=I_{MO}(\tilde{p}_{h})$. Then it holds
\begin{equation}\label{auxiliary problem 3}
\{\sum\limits_{K\in\mathcal{T}_{h}}
||S^{1/2}\nabla(\tilde{p}_{h}-s)||_{K}^{2}\}^{1/2}\leq
c_{14}\{\sum\limits_{\sigma\in\varepsilon_{h}}\Lambda_{\sigma}h_{\sigma}||[\gamma_{{\bf
t}_{\sigma}}(S^{-1}{\bf u}_{h})]||_{\sigma}^{2}\}^{1/2},
\end{equation}
where $\Lambda_{\sigma}$ is given in Section 4, and ${\bf
t}_{\sigma}$ denotes the unit tangent vector along $\sigma$.
\end{lemma}
\begin{proof} From the estimate (\ref{Oswald interpolation 1}) we have
\begin{equation}\label{LEMMA 6.3.1}
||\nabla(\tilde{p}_{h}-s)||_{K}^{2}\leq
c_{13}\sum\limits_{\sigma,\sigma\cap K\neq\emptyset}
h_{\sigma}^{-1}||[\tilde{p}_{h}]||_{\sigma}^{2},\ \ {\rm for\ all}\
K\in\mathcal{T}_{h},
\end{equation}
where $\sigma\cap K\neq\emptyset$ when $\sigma$ contains a vertex of $K$.

Since the mean of $\tilde{p}_{h}$ over interior side is continuous
and its mean on exterior side vanishes, i.e.,
$\displaystyle\int_{\sigma}[\tilde{p}_{h}]ds=0$ for all
$\sigma\in\varepsilon_{h}$, by Poincar\'{e} inequality it holds
\begin{equation}\label{LEMMA 6.3.2}
||[\tilde{p}_{h}]||_{\sigma}=||[\tilde{p}_{h}]-\int_{\sigma}[\tilde{p}_{h}]||_{\sigma}\leq
c_{15}h_{\sigma}||\gamma_{{\bf
t}_{\sigma}}(\nabla([\tilde{p}_{h}]))||_{\sigma}.
\end{equation}
The postprocessing (\ref{postprocessed variable 1}) indicates
\begin{equation}\label{LEMMA 6.3.3}
\gamma_{{\bf t}_{\sigma}}(\nabla([\tilde{p}_{h}]))=-[\gamma_{{\bf
t}_{\sigma}}(S^{-1}{\bf u}_{h})], \ \ {\rm for\ all}\
\sigma\in\varepsilon_{h}.
\end{equation}
A combination of (\ref{LEMMA 6.3.1})-(\ref{LEMMA 6.3.3}) yields
\begin{equation}\label{LEMMA 6.3.4}
||S^{1/2}\nabla(\tilde{p}_{h}-s)||_{K}^{2}\leq
c_{13}c_{15}C_{S,K}\sum\limits_{\sigma,\sigma\cap K\neq\emptyset}
h_{\sigma}||[\gamma_{{\bf t}_{\sigma}}(S^{-1}{\bf
u}_{h})]||_{\sigma}^{2}.
\end{equation}
Summing (\ref{LEMMA 6.3.4}) over each element $K$,   noticing
that the number of summation over a side $\sigma\in\varepsilon_{h}$
is bounded by a positive constant $c_{17}$, and combining the
definition of $\Lambda_{\sigma}$, we obtain
\begin{equation}\label{LEMMA 6.3.5}
\begin{array}{lll}
\displaystyle\sum\limits_{K\in\mathcal{T}_{h}}||S^{1/2}\nabla(\tilde{p}_{h}-s)||_{K}^{2}&\leq&
\displaystyle
c_{13}c_{15}\sum\limits_{K\in\mathcal{T}_{h}}C_{S,K}\sum\limits_{\sigma,\sigma\cap
K\neq\emptyset} h_{\sigma}||[\gamma_{{\bf
t}_{\sigma}}(S^{-1}{\bf u}_{h})]||_{\sigma}^{2}\vspace{2mm}\\
&\leq&\displaystyle
c_{13}c_{15}c_{16}\sum\limits_{\sigma\in\varepsilon_{h}}\Lambda_{\sigma}h_{\sigma}
||[\gamma_{{\bf t}_{\sigma}}(S^{-1}{\bf u}_{h})]||_{\sigma}^{2}.
\end{array}
\end{equation}
The desired result (\ref{auxiliary problem 3}) with
$c_{14}:=c_{13}c_{15}c_{16}$ follows from (\ref{LEMMA 6.3.5}).
\end{proof}

\begin{remark}\label{new 11} {\it The node with respect to which the quasi-monotone condition
 is violated is called singular node (cf. \cite{Petzoldt}).
 We can derive an alternative   form of (\ref{LEMMA 6.3.5}) as following:
\begin{equation*}
\sum\limits_{K\in\mathcal{T}_{h}}||S^{1/2}\nabla(\tilde{p}_{h}-s)||_{K}^{2}\leq
c_{14}\sum\limits_{K\in\mathcal{T}_{h}}\xi_{K}^{2},
\end{equation*}
where
\begin{equation*}
\xi_{K}^{2} :=\left \{ \begin{array}{ll}
  \sum\limits_{\sigma\in\varepsilon_{K}}h_{\sigma} ||[\gamma_{{\bf
t}_{\sigma}}(S^{-1/2}{\bf u}_{h})]||_{\sigma}^{2},\ \  & \mbox{if}\;  \ \ \ K\ {\rm has\ no\ singular\ nodes},\\
   \sum\limits_{\sigma\in\varepsilon_{K}}C_{S,\omega_{K}}h_{\sigma}
||[\gamma_{{\bf t}_{\sigma}}(S^{-1}{\bf u}_{h})]||_{\sigma}^{2},\ \
& \mbox{if}\; \ \ K\ {\rm includes\ a\ singular\ node}
 \end{array}\right.
\end{equation*}
with
 $C_{S,\omega_{K}}:=\max_{K'\in\tilde{\omega}_{K}}C_{S,K'}$.}
\end{remark}

\begin{lemma}\label{dual problem 2}
Let $\Lambda_{{\bf w},r,K}$ be
  the same  as in (\ref{new 3}) and $s:=I_{MO}(\tilde{p}_{h})$. Then it
holds
\begin{equation}\label{auxiliary problem 4}
\{\sum\limits_{K\in\mathcal{T}_{h}} c_{{\bf
w},r,K}||\tilde{p}_{h}-s||_{K}^{2}\}^{1/2}\leq\displaystyle
c_{17}\{\sum\limits_{K\in\mathcal{T}_{h}}\Lambda_{{\bf
w},r,K}h_{K}^{2}||S^{-1}{\bf u}_{h}||_{K}^{2}\}^{1/2}.
\end{equation}

\end{lemma}
\begin{proof} Following the line of the proof of THEOREM 2.2 in \cite{Karakashian},
we obtain
\begin{equation}\label{LEMMA 6.4.1}
||\tilde{p}_{h}-s||_{K}^{2}\leq c_{18}\sum\limits_{\sigma:\sigma\cap
K\neq\emptyset} h_{\sigma}||[\tilde{p}_{h}]||_{\sigma}^{2}.
\end{equation}

Let $\tilde{p}_{\sigma} :=<1,\tilde{p}_{h}>_{\sigma}/|\sigma|$
denote the mean of the postprocessed scalar variable $\tilde{p}_{h}$
over a side $\sigma\in\varepsilon_{h}$. From the trace theory and
generalized Friedrichs inequality (\ref{Friedrichs inequality}), we
obtain
\begin{equation}\label{Oswald interpolation 2}
||[\tilde{p}_{h}]||_{\sigma}\leq
c_{19}h_{\sigma}^{1/2}||\nabla_{h}\tilde{p}_{h}||_{\omega_{\sigma}}.
\end{equation}

A combination of (\ref{LEMMA 6.4.1}), (\ref{Oswald interpolation 2})
and the postprocessing (\ref{postprocessed variable 1}) yields that
\begin{equation}\label{LEMMA 6.4.2}
||\tilde{p}_{h}-s||_{K}^{2}\leq c_{20}\sum\limits_{\sigma,\sigma\cap
K\neq\emptyset} h_{\sigma}^{2}||S^{-1}{\bf
u}_{h}||_{\omega_{\sigma}}^{2}.
\end{equation}
Summing (\ref{LEMMA 6.4.2}) over each element $K$, noticing that the
mesh is local quasi-uniform, and combining the definition of
$\Lambda_{{\bf w},r,K}$, we finally get
\begin{equation*}
\begin{array}{lll}
\displaystyle\sum\limits_{K\in\mathcal{T}_{h}} c_{{\bf
w},r,K}||\tilde{p}_{h}-s||_{K}^{2}&\leq&\displaystyle
c_{20}\sum\limits_{K\in\mathcal{T}_{h}}c_{{\bf
w},r,K}\sum\limits_{\sigma,\sigma\cap K\neq\emptyset}
h_{\sigma}^{2}||S^{-1}{\bf u}_{h}||_{\omega_{\sigma}}^{2}\vspace{2mm}\\
&\leq&c_{17}^{2}\sum\limits_{K\in\mathcal{T}_{h}}\Lambda_{{\bf
w},r,K}h_{K}^{2}||S^{-1}{\bf u}_{h}||_{K}^{2}.
\end{array}
\end{equation*}
\end{proof}

\begin{remark} {\rm(Alternative form)}\ {\it For a side
$\sigma\in\varepsilon_{h}$, we denote
$
\Lambda_{{\bf w},r,\sigma}
:=\max_{K,K\cap\sigma\neq\emptyset}\{c_{{\bf w},r,K}\}.
$
   A combination of (\ref{LEMMA 6.3.2}), (\ref{LEMMA 6.3.3}) and
(\ref{LEMMA 6.4.1}) yields
\begin{equation}\label{LEMMA 6.5.1}
||\tilde{p}_{h}-s||_{K}^{2}\leq c_{21}\sum\limits_{\sigma,\sigma\cap
K\neq\emptyset} h_{\sigma}^{2}||[\gamma_{{\bf
t}_{\sigma}}(S^{-1}{\bf u}_{h})]||_{\sigma}^{2},
\end{equation}
which leads to an alternative form of the estimate (\ref{auxiliary problem 4}),
\begin{equation*}
\{\sum\limits_{K\in\mathcal{T}_{h}} c_{{\bf
w},r,K}||\tilde{p}_{h}-s||_{K}^{2}\}^{1/2}\leq
c_{22}\{\sum\limits_{\sigma\in\varepsilon_{h}} \Lambda_{{\bf
w},r,\sigma}h_{\sigma}^{2}||[\gamma_{{\bf t}_{\sigma}}(S^{-1}{\bf
u}_{h})]||_{\sigma}^{2}\}^{1/2}.
\end{equation*}
This inequality shows that the term
$\{\sum\limits_{K\in\mathcal{T}_{h}} c_{{\bf
w},r,K}||\tilde{p}_{h}-s||_{K}^{2}\}^{1/2}$ can be absorbed into
$
\{\sum\limits_{K\in\mathcal{T}_{h}}
||S^{1/2}\nabla(\tilde{p}_{h}-s)||_{K}^{2}\}^{1/2}
$ 
when $\Lambda_{{\bf w},r,\sigma}h_{\sigma}\leq\Lambda_{\sigma}$.}
\end{remark}

The following corollary is a combined result of LEMMAs
\ref{auxialiary problem 2}-\ref{dual problem 2}.

\begin{corollary}\label{du-xu1}\ {\it Let $\eta_{NC,K}$ be  defined as in (\ref{new 3}) and $s:=I_{MO}(\tilde{p}_{h})$. Then it holds
\begin{equation}\label{corollary 6.5}
\displaystyle|||\tilde{p}_{h}-s|||_{\Omega}\leq
c_{23}\{\sum\limits_{K\in\mathcal{T}_{h}}\eta_{NC,K}^{2}\}^{1/2}.
\end{equation}}
\end{corollary}
\begin{lemma}\label{du-xu2}\ {\rm(Convection estimator.)}\ {\it Let
$T_{C}(\varphi,s)$  be  defined as in
(\ref{covection component 1}) with $|||\varphi|||_{\Omega}=1$ and $s:=I_{MO}(\tilde{p}_{h})$, and $\eta_{C,K}$ be defined as in
(\ref{new 4}). Then it holds
\begin{equation}\label{Convection estimates 1}
\displaystyle T_{C}(\varphi,s)\leq
c_{24}\{\sum\limits_{K\in\mathcal{T}_{h}}\eta_{C,K}^{2}\}^{1/2}.
\end{equation}}
\end{lemma}
\begin{proof}
 By  triangle inequality and H\"{o}lder inequality we
obtain
\begin{equation}\label{convection estimates 2}
\begin{array}{lll}
T_{C}(\varphi,s)&\leq&\displaystyle\sum\limits_{K\in\mathcal{T}_{h}}\{C_{{\bf
w},K}||\nabla(\tilde{p}_{h}-s)||_{K}||\varphi||_{K}+\frac{1}{2}C_{\nabla\cdot{\bf
w},K}||\tilde{p}_{h}-s||_{K}||\varphi||_{K}\}\vspace{2mm}\\
&\leq&\displaystyle
\{\sum\limits_{K\in\mathcal{T}_{h}}
(\frac{C_{{\bf
w},K}^{2}}{c_{S,K}c_{{\bf
w},r,K}}||S^{1/2}\nabla(\tilde{p}_{h}-s)||_{K}^{2}+\frac{C_{\nabla\cdot{\bf
w},K}^{2}}{4c_{{\bf
w},r,K}}||\tilde{p}_{h}-s||_{K}^{2})\}^{1/2}.
\end{array}
\end{equation}
Apply (\ref{LEMMA 6.3.4}) and (\ref{LEMMA 6.4.2}) to the
inequality (\ref{convection estimates 2}), and combine the
definitions of $\lambda_{{\bf w},\sigma}$ and
$\Lambda_{\nabla\cdot{\bf w},K}$, we then arrive at
\begin{equation}\label{LEMMA 6.6.1}
\begin{array}{lll}
T_{C}(\varphi,s)&\leq&\displaystyle
c_{25}\{\sum\limits_{K\in\mathcal{T}_{h}} (\frac{C_{{\bf
w},K}^{2}C_{S,K}}{c_{S,K}c_{{\bf
w},r,K}}\sum\limits_{\sigma:\sigma\cap
K\neq\emptyset}h_{\sigma}||[\gamma_{{\bf t}_{\sigma}}(S^{-1}{\bf
u}_{h})]||_{\sigma}^{2}\vspace{2mm}\\
&\ &\displaystyle+\frac{C_{\nabla\cdot{\bf w},K}^{2}}{4c_{{\bf
w},r,K}}\sum\limits_{\sigma:\sigma\cap K\neq\emptyset}h_{\sigma}^{2}||S^{-1}{\bf u}_{h}||_{\omega_{\sigma}}^{2})\}^{1/2}\vspace{2mm}\\
&\leq&\displaystyle
c_{25}\{\sum\limits_{\sigma\in\varepsilon_{h}}\lambda_{{\bf
w},\sigma}^{2}h_{\sigma} ||[\gamma_{{\bf t}_{\sigma}}(S^{-1}{\bf
u}_{h})]
||_{\sigma}^{2}+\sum\limits_{K\in\mathcal{T}_{h}}\Lambda_{\nabla\cdot{\bf
w},K}^{2}h_{K}^{2}||S^{-1}{\bf u}_{h}||_{K}^{2}\}^{1/2}
\end{array}
\end{equation}
Since the modified Oswald interpolation
$s=\mathcal{I}_{MO}(\tilde{p}_{h})$ preserves the mean of
$\tilde{p}_{h}$ on the side, and ${\bf w}\cdot{\bf n}$ is constant
over a side, it holds
\begin{equation*}
(\nabla\cdot((\tilde{p}_{h}-s){\bf w}),\varphi_{K})_{K}=
<(\tilde{p}_{h}-s){\bf w}\cdot{\bf n},\varphi_{K}>_{\partial K}=0,
\end{equation*}
where $\varphi_{K}$ is the mean of $\varphi$ over $K$. Write $v
:=\tilde{p}_{h}-s$, then we have
\begin{equation}\label{LEMMA 6.6.2}
(\nabla\cdot(v{\bf w})-1/2v\nabla\cdot{\bf w},\varphi)_{K}=(\nabla v\cdot{\bf w},\varphi-\varphi_{K})_{K}
+(1/2v\nabla\cdot{\bf w},\varphi)_{K}-(v\nabla\cdot{\bf w},\varphi_{K})_{K}.
\end{equation}

A combination of  (\ref{LEMMA
6.6.2}), (\ref{Poicare inequality}), (\ref{LEMMA 6.3.4}),
(\ref{LEMMA 6.4.2}) and H\"{o}lder inequality yields
\begin{equation}\label{LEMMA 6.6.3}
\begin{array}{lll}
T_{C}(\varphi,s)&\leq&\displaystyle\sum\limits_{K\in\mathcal{T}_{h}}(\frac{h_{K}C_{{\bf w},K}}{\sqrt{c_{S,K}}}||S^{1/2}\nabla (\tilde{p}_{h}-s)||_{K}
+\frac{3C_{\nabla\cdot{\bf w},K}}{2\sqrt{c_{{\bf w},r,K}}}||\tilde{p}_{h}-s||_{K})|||\varphi|||_{K}\vspace{2mm}\\
&\leq&\displaystyle
c_{26}\{\sum\limits_{K\in\mathcal{T}_{h}}\frac{h_{K}^{2}C_{{\bf
w},K}^{2}}{c_{S,K}} \sum\limits_{\sigma,\sigma\cap
K\neq\emptyset}h_{\sigma}||[\gamma_{{\bf t}_{\sigma}}(S^{-1}{\bf
u}_{h})]
||_{\sigma}^{2}\vspace{2mm}\\
&\
&\displaystyle+\sum\limits_{K\in\mathcal{T}_{h}}\frac{C_{\nabla\cdot{\bf
w},K}^{2}}{c_{{\bf w},r,K}} \sum\limits_{\sigma,\sigma\cap
K\neq\emptyset}h_{\sigma}^{2}||S^{-1}{\bf
u}_{h}||_{\omega_{\sigma}}^{2}\}^{1/2}.
\end{array}
\end{equation}
This estimate, together with the definitions of $p_{{\bf w},\sigma}$ and $\Lambda_{\nabla\cdot{\bf w},K}$, indicates
$T_{C}(\varphi,s)$ from (\ref{LEMMA 6.6.3})
\begin{equation}\label{LEMMA 6.6.4}
T_{C}(\varphi,s)\leq\displaystyle
c_{26}\{\sum\limits_{\sigma\in\varepsilon_{h}}p_{{\bf
w},\sigma}^{2}h_{\sigma} ||[\gamma_{{\bf t}_{\sigma}}(S^{-1}{\bf
u}_{h})]
||_{\sigma}^{2}+\sum\limits_{K\in\mathcal{T}_{h}}\Lambda_{\nabla\cdot{\bf
w},K}^{2}h_{K}^{2}||S^{-1}{\bf u}_{h}||_{K}^{2}\}^{1/2}.
\end{equation}
The desired result (\ref{Convection estimates 1}) follows from (\ref{LEMMA 6.6.1}) and
(\ref{LEMMA 6.6.4}) with
$c_{24}=\max\{c_{25},c_{26}\}$.
\end{proof}\\

{\bf Proof of THEOREMs \ref{global error estimate}-\ref{Global
upwind}}: For the centered mixed scheme, the desired result
(\ref{global error estimate 1}) follows from LEMMA \ref{abstract
error estimate 1}, LEMMA \ref{residual estimator (global)},
Corollary \ref{du-xu1}, LEMMA \ref{du-xu2} with the positive
constant $c_{1}=2\sqrt{2}\max(1,c_{8},c_{23},c_{24})$. For the
upwind-weighted mixed scheme, the assertion (\ref{global error
estimate 2}) follows from   LEMMA \ref{residual estimator
(global)}-\ref{upwind estimator (global)}, Corollary \ref{du-xu1},
 LEMMA \ref{du-xu2} and LEMMA \ref{abstract error estimate 1}with
 $c_{2}=\sqrt{10}\max(1,c_{8},c_{9},c_{23},c_{24})$.
\\

\begin{remark}\
{\rm(Two approaches in   a posteriori error
analysis)}
\ {\it There are usually two approaches in literature in the a posteriori error
analysis. One is directly based on the  solution of the
discretization scheme, the other one is based on the postprocessed approximation. Seemingly,
these two approaches are fully different. Our analysis establishes a link between them, i.e. 
 a posteriori error estimates  based on the discretization solution
 can be derived with the help of the postprocessing
technique. In doing so, one can avoid 
the use of Helmholtz decomposition of the stress variable which is required in traditional a posteriori error analysis for mixed finite elements.}
\end{remark}

\begin{remark}\ {\rm(Pure diffusion problem)}\ {\it When ${\bf w}=r=0$, the model (\ref{convection-diffufsion-equations1}) is reduced to a pure
diffusion problem. In this case,
the fact that $-\nabla\cdot(S_{K}\nabla\tilde{p}_{h}|_{K})=\nabla\cdot{\bf u}_{h}|_{K}=f_{K}$ for all
$K\in\mathcal{T}_{h}$ with $f_{K}$  the mean value of $f$ over $K$ indicates
\begin{equation}\label{remark 6.9.1}
\eta_{D,K}=0,\ \eta_{C,K}=0, \ \eta_{R,K}^{2} =\frac{h_{K}^{2}}{c_{S,K}}||f-f_{K}||_{K}^{2},\
\eta_{NC,K}^{2}
=\sum\limits_{\sigma\in\varepsilon_{K}}\delta_{\sigma}\Lambda_{\sigma}h_{\sigma}
||[\gamma_{{\bf t}_{\sigma}}(S^{-1}{\bf u}_{h})]||_{\sigma}^{2}.
\end{equation}
Thus the a posteriori error
estimate (\ref{global error estimate 1}) is reduced to
\begin{equation}\label{remark 6.9.3}
\mathcal{E}\leq
c_{27}\{\sum\limits_{K\in\mathcal{T}_{h}}(\frac{h_{K}^{2}}{c_{S,K}}||f-f_{K}||_{K}^{2}+
\sum\limits_{\sigma\in\varepsilon_{K}}\delta_{\sigma}\Lambda_{\sigma}h_{\sigma}||[\gamma_{{\bf
t}_{\sigma}}(S^{-1}{\bf u}_{h})]||_{\sigma}^{2})\}^{1/2}
\end{equation}
with $
\mathcal{E} =\{\sum\limits_{K\in\mathcal{T}_{h}}
||S^{-1/2}({\bf u}-{\bf u}_{h})||_{K}^{2}\}^{1/2}.
$
 In addition,  Remark \ref{new 11} implies an alternative estimate
\begin{equation}\label{new 12}
\mathcal{E}\leq
c_{27}\{\sum\limits_{K\in\mathcal{T}_{h}}(\frac{h_{K}^{2}}{c_{S,K}}||f-f_{K}||_{K}^{2}+
\xi_{K}^{2}\}^{1/2}.
\end{equation}
Note that being an oscillation term,
the first term in the right side of (\ref{remark
6.9.3}) or (\ref{new 12}) may not be computed in practice.}
\end{remark}

\begin{remark}\label{du-xu3}\ {\rm(A posteriori error estimate of
divergence of the stress variable.)}\ {\it
The continuous weak formulation of
(\ref{convection-diffufsion-equations1}) reads as: Find $({\bf
u},p)\in{\bf{H}({\rm div},\Omega)}\times L^{2}(\Omega)$ such that
\begin{equation*}
(S^{-1}{\bf u},{\bf v})_{\Omega}-(p,\nabla\cdot{\bf v})_{\Omega}=0\
\ \ {\rm for\ all}\ {\bf v}\in{\bf{H}({\rm div},\Omega)},
\end{equation*}
\begin{equation}\label{remarke 6.5}
(\nabla\cdot{\bf u},\varphi)_{\Omega}-(S^{-1}{\bf u}\cdot{\bf
w},\varphi)_{\Omega}+((r+\nabla\cdot{\bf
w})p,\varphi)_{\Omega}=(f,\varphi)_{\Omega}\ \ \ {\rm for\ all}\ \
\varphi\in L^{2}(\Omega).
\end{equation}
Notice that (\ref{remarke 6.5}) can be equivalently written as: For
each $K\in\mathcal{T}_{h}$
\begin{equation}\label{remark 6.5a}
(\nabla\cdot{\bf u},\varphi)_{K}-(S^{-1}{\bf u}\cdot{\bf
w},\varphi)_{K}+((r+\nabla\cdot{\bf
w})p,\varphi)_{K}=(f,\varphi)_{K}\ \ \ {\rm for\ all}\ \ \varphi\in
L^{2}(K).
\end{equation}
Meanwhile, the centered mixed finite element scheme (\ref{centered
mixed scheme 2}) can be equivalently written as: For every
$K\in\mathcal{T}_{h}$
\begin{equation}\label{remark 6.5b}
(\nabla\cdot{\bf u}_{h},\varphi)_{K}-(S^{-1}{\bf u}_{h}\cdot{\bf
w},\varphi)_{K}+((r+\nabla\cdot{\bf
w})p_{h},\varphi)_{K}=(f,\varphi)_{K}\ \ \ {\rm for\ all}\ \
\varphi\in P_{0}(K).
\end{equation}
Let $\overline{R}_{K}$ denote the mean of the elementwise residual
\begin{equation*}
R_{K} :=f-\nabla\cdot{\bf u}_{h}+(S^{-1}{\bf u}_{h})\cdot{\bf
w}-(r+\nabla\cdot{\bf w})p_{h}
\end{equation*}
over $K\in\mathcal{T}_{h}$, and set $0\leq\iota\leq1$. We define the data
oscillation ${\rm osc}_{h}$ as
\begin{equation*}
{\rm osc}_{h}:
=\{\sum\limits_{K\in\mathcal{T}_{h}}h_{K}^{2\iota}||R_{K}-\overline{R}_{K}||_{K}^{2}
\}^{1/2}.
\end{equation*}
For any $\varphi\in L^{2}(K)$, let $\varphi_{K}$ denote the
mean of $\varphi$ over $K\in\mathcal{T}_{h}$, then a combination of
(\ref{remark 6.5a}) and (\ref{remark 6.5b}) yields
\begin{equation*}
\begin{array}{lll}
&\ &(\nabla\cdot({\bf u}-{\bf u}_{h}),\varphi)_{K}=(\nabla\cdot{\bf
u},\varphi)_{K}-(\nabla\cdot{\bf
u}_{h},\varphi-\varphi_{K})_{K}-(\nabla\cdot{\bf
u}_{h},\varphi_{K})_{K}\vspace{2mm}\\
&=&(R_{K},\varphi-\varphi_{K})_{K}+(S^{-1}({\bf u}-{\bf
u}_{h})\cdot{\bf w},\varphi)_{K}-((r+\nabla\cdot{\bf
w})(p-p_{h}),\varphi)_{K}\vspace{2mm}\\
&=&(R_{K}-\overline{R}_{K},\varphi-\varphi_{K})_{K}+(S^{-1}({\bf
u}-{\bf u}_{h})\cdot{\bf w},\varphi)_{K}-((r+\nabla\cdot{\bf
w})(p-p_{h}),\varphi)_{K}\vspace{2mm}\\
&\leq&(||R_{K}-\overline{R}_{K}||_{K}+||S^{-1}({\bf u}-{\bf
u}_{h})||_{K}||{\bf w}||_{L^{\infty}(K)}+C_{{\bf
w},r,K}||p-p_{h}||_{K})||\varphi||_{K}.
\end{array}
\end{equation*}
This means that
\begin{equation}\label{remark6.5c}
\begin{array}{lll}
&\ &||\nabla\cdot({\bf u}-{\bf u}_{h})||_{K}=\sup_{\varphi\in
L^{2}(K),\varphi\neq0}\frac{(\nabla\cdot({\bf u}-{\bf
u}_{h}),\varphi)_{K}}{||\varphi||_{L^{2}(K)}}\vspace{2mm}\\
&\ &\ \ \leq||R_{K}-\overline{R}_{K}||_{K}+\frac{||{\bf
w}||_{L^{\infty}(K)}}{\sqrt{c_{S,K}}}||S^{-1/2}({\bf u}-{\bf
u}_{h})||_{K}+C_{{\bf w},r,K}||p-p_{h}||_{K}.
\end{array}
\end{equation}
From (\ref{remark6.5c}) and (\ref{global error estimate 1}) we
obtain the following a posteriori error estimate of the divergence of the
stress variable for the centered mixed finite element scheme:
\begin{equation*}
||h^{\iota}\nabla\cdot({\bf u}-{\bf u}_{h})||\leq
_{28}\{\{\sum\limits_{K\in\mathcal{T}_{h}}(\eta_{R,K}^{2}+\eta_{C,K}^{2}+
\eta_{NC,K}^{2}+\eta_{D,K}^{2})\}^{1/2}\beta_{{\rm c}}+{\rm
osc}_{h}\},
\end{equation*}
where the constant $
\beta_{{\rm c}} :=\max_{K\in\mathcal{T}_{h}}\max\{\frac{||{\bf
w}||_{L^{\infty}(K)}}{\sqrt{c_{S,K}}}h_{K}^{\iota},\frac{C_{{\bf
w},r,K}}{\sqrt{c_{{\bf w},r,K}}}h_{K}^{\iota}\}.
$ }
\end{remark}

\section{Analysis of local efficiency}\
Using  standard arguments we easily derive 
lemmas \ref{efficiency residual }-\ref{efficiency jump}.
\begin{lemma}\label{efficiency residual }
Denote $v :=f-\nabla\cdot{\bf u}_{h}+(S^{-1}{\bf u}_{h})\cdot{\bf
w}-(r+\nabla\cdot{\bf w})p_{h}$, and let $\mathcal{E}_{K}$ be the
local error for the stress and displacement defined in (\ref{new-1}). Under Assumption
$(D5)$ for $f$,  it holds
\begin{equation}\label{efficiency residual 1}
h_{K}||v||_{K}\leq c_{29}\max\{\sqrt{C_{S,K}}+\frac{C_{{\bf
w},K}}{\sqrt{c_{S,K}}}h_{K},\frac{C_{{\bf w},r,K}}{\sqrt{c_{{\bf
w},r,K}}}h_{K}\}\mathcal{E}_{K}.
\end{equation}
\end{lemma}
\begin{lemma}\label{efficiency jump}
It holds
\begin{equation}\label{efficiency jump 1}
h_{\sigma}^{1/2}||[\gamma_{{\bf t}_{\sigma}}(S^{-1}{\bf u}_{h})]
||_{\sigma}\leq c_{30}c_{\omega_{\sigma}}||S^{-1/2}({\bf u}-{\bf
u}_{h})||_{\omega_{\sigma}}.
\end{equation}
\end{lemma}
\begin{lemma}\label{efficiency volumn }
It holds
\begin{equation}\label{efficiency volumn 1}
h_{K}||S^{-1}{\bf u}_{h}||_{K}\leq
c_{31}\max\{\frac{h_{K}}{\sqrt{c_{S,K}}},\frac{1}{\sqrt{c_{{\bf
w},r,K}}}\}\mathcal{E}_{K}.
\end{equation}
\end{lemma}
\begin{proof} For all $K\in\mathcal{T}_{h}$, let $\psi_{K}$ denote the bubble
function on $K$ with zero boundary values on $K$ and
$0\leq\psi_{K}\leq1$. The relation $S^{-1}{\bf
u}_{h}|_{K}=(S^{-1}{\bf u}_{h}+\nabla p_{h})|_{K}$ for all
$K\in\mathcal{T}_{h}$ shows
\begin{equation}\label{efficiency volumn 2}
\begin{array}{lll}
||S^{-1}{\bf u}_{h}||_{K}^{2}&=&||S^{-1}{\bf u}_{h}+\nabla
p_{h}||_{K}^{2}\vspace{2mm}\\
&\leq&c_{32}||\psi_{K}^{1/2}(S^{-1}{\bf u}_{h}+\nabla
p_{h})||_{K}^{2}\vspace{2mm}\\
&=&c_{32}(\psi_{K}S^{-1}{\bf u}_{h},S^{-1}{\bf u}_{h}+\nabla
p_{h})_{K}\vspace{2mm}\\
&=&c_{32}\{(\psi_{K}S^{-1}{\bf u}_{h},S^{-1}({\bf u}_{h}-{\bf
u}))_{K}\vspace{2mm}\\
&\ &+(\psi_{K}S^{-1}{\bf u}_{h},S^{-1}{\bf u}+\nabla p_{h})_{K}\}.
\end{array}
\end{equation}
Integration by parts implies
\begin{equation}\label{efficiency volumn 3}
\begin{array}{lll}
(\psi_{K}S^{-1}{\bf u}_{h},S^{-1}{\bf u}+\nabla
p_{h})_{K}&=&(\psi_{K}S^{-1}{\bf
u}_{h},\nabla(p_{h}-p))_{K}\vspace{2mm}\\
&=&-(\nabla\cdot(\psi_{K}S^{-1}{\bf u}_{h}),p_{h}-p)_{K}.
\end{array}
\end{equation}
A combination of (\ref{efficiency volumn 2}), (\ref{efficiency
volumn 3}) and  inverse inequality imply
\begin{equation}\label{efficiency volumn 4}
||S^{-1}{\bf u}_{h}||_{K}^{2}\leq
c_{32}\{\frac{1}{\sqrt{c_{S,K}}}||S^{-1/2}({\bf u}-{\bf
u}_{h})||_{K}+h_{K}^{-1}||p-p_{h}||_{K}\}||S^{-1}{\bf u}_{h}||_{K}.
\end{equation}
The desired result (\ref{efficiency volumn 1}) then follows with
$c_{31}:=\sqrt{2}c_{32}$.
\end{proof}

\begin{lemma}\label{efficiency upwind-scheme}
{ It holds
\begin{equation*}
 h_{K}^{1/2}||\hat{\hat{p}}||_{\sigma}\leq \left\{\begin{array}{ll}
  c_{33}|\sigma|^{-1/2}(1/2-\nu_{\sigma})\left(h_{K}||S^{-1}{\bf u}_{h}||_{K}+
  h_{L}||S^{-1}{\bf u}_{h}||_{L}\right) & \text{ if } \sigma=\bar{K}\cap \bar{L},\\
  c_{33}|\sigma|^{-1/2}(1-\nu_{\sigma})h_{K}||S^{-1}{\bf u}_{h}||_{K} & \text{ if } \sigma\in\varepsilon_{K}\cap\varepsilon_{h}^{\rm ext}.
\end{array} \right.
\end{equation*}
}
\end{lemma}
\begin{proof} If $\sigma=\bar{K}\cap \bar{L}$, from
$\displaystyle\int_{\sigma}\tilde{p}_{h}|_{K}ds=\int_{\sigma}\tilde{p}_{h}|_{L}ds$
we have
\begin{equation*}
 \int_{\sigma}|\hat{\hat{p}}_{\sigma}|=\left \{ \begin{array}{ll}
  \displaystyle\int_{\sigma}(1/2-\nu_{\sigma})\left((p_{h}-\tilde{p}_{h})|_{K}-(p_{h}-
  \tilde{p}_{h})|_{L}\right)\ \  & \mbox{if}\;  \ \ p_{K}\geq p_{L},\vspace{2mm}\\
 \displaystyle\int_{\sigma}(1/2-\nu_{\sigma})\left((p_{h}-\tilde{p}_{h})|_{L}-(p_{h}-
  \tilde{p}_{h})|_{K}\right)\ \  & \mbox{if}\;  \ \ p_{K}<p_{L}.
 \end{array}\right.
\end{equation*}
This relation, together with trace theorem and the
postprocessing (\ref{postprocessed variable 2}), indicates
\begin{equation*}
\begin{array}{lll}
\displaystyle\int_{\sigma}|\hat{\hat{p}}_{\sigma}|&\leq&(1/2-
\nu_{\sigma})(||p_{h}-\tilde{p}_{h}||_{\partial
K}+||p_{h}-\tilde{p}_{h}||_{\partial L})\vspace{2mm}\\
&\leq&c_{33}(1/2-\nu_{\sigma})(h_{K}^{1/2}||\nabla\tilde{p}_{h}||_{K}+
h_{L}^{1/2}||\nabla\tilde{p}_{h}||_{L}),
\end{array}
\end{equation*}
which, together with the local shape
regularity of elements and the postprocessing (\ref{postprocessed
variable 1}), implies
\begin{equation*}
\begin{array}{lll}
h_{K}^{1/2}||\hat{\hat{p}}||_{\sigma}&=&\displaystyle
h_{K}^{1/2}|\hat{\hat{p}}||\sigma|^{1/2}
=h_{K}^{1/2}|\sigma|^{-1/2}\int_{\sigma}|\hat{\hat{p}}_{\sigma}|ds\vspace{2mm}\\
&\leq&c_{33}|\sigma|^{-1/2}(1/2-\nu_{\sigma})\left(h_{K}||S^{-1}{\bf
u}_{h}||_{K}+h_{L}||S^{-1}{\bf u}_{h}||_{L}\right).
\end{array}
\end{equation*}

If
$\sigma\in\varepsilon_{K}\cap\varepsilon_{h}^{\rm ext}$, from
$\displaystyle\int_{\sigma}\tilde{p}_{h}|_{K}ds=0$ and
$\nu_{\sigma}\leq1/2$ the second assertion of the lemma follows.
\end{proof}\\

{\bf Proof of THEOREM \ref{new 5}}: From  the definition of
$\alpha_{*,K}$ in this theorem,  LEMMA \ref{efficiency
volumn } shows
\begin{equation}\label{efficiency displacement and residual}
\eta_{D,K}\leq c_{34}\alpha_{*,K}\mathcal{E}_{K}.
\end{equation}
Denote $v :=f-\nabla\cdot{\bf u}_{h}+(S^{-1}{\bf u}_{h})\cdot{\bf
w}-(r+\nabla\cdot{\bf w})p_{h}$, and then it holds
\begin{equation}\label{efficiency displacement residual 2}
\eta_{R,K}\leq\alpha_{K}||v||_{K}+\beta_{K}||S^{-1}{\bf
u}_{h}||_{K}\leq\frac{h_{K}}{\sqrt{c_{S,K}}}||v||_{K}+C_{{\bf
w},r,K}h_{K}\frac{h_{K}}{\sqrt{c_{S,K}}}||S^{-1}{\bf u}_{h}||_{K}.
\end{equation}
A combination of (\ref{efficiency displacement residual 2}), LEMMA
\ref{efficiency residual } and LEMMA \ref{efficiency volumn }, leads
to
\begin{equation}\label{efficiency displacement residual 1}
\eta_{R,K}\leq c_{35}\alpha_{*,K}\mathcal{E}_{K}.
\end{equation}
The desired result (\ref{new 6}) follows from (\ref{efficiency
displacement and residual}) and (\ref{efficiency displacement
residual 1}) with $c_{3}=\sqrt{2}\max\{c_{34},c_{35}\}$.\\

{\bf Proof of THEOREM \ref{new 7}}: Notice that
\begin{equation}\label{efficiency conforming convection 3}
\eta_{NC,K}\leq\displaystyle \sqrt{\Lambda_{{\bf
w},r,K}}h_{K}||S^{-1}{\bf u}_{h}||_{K}
+\sum\limits_{\sigma\in\varepsilon_{K}}\Lambda_{\sigma}^{1/2}h_{\sigma}^{1/2}||[\gamma_{{\bf
t}_{\sigma}}(S^{-1}{\bf u}_{h})]||_{\sigma}
\end{equation}
and
\begin{equation}\label{efficiency conforming convecton 4}
\eta_{C,K}\leq\displaystyle \Lambda_{\nabla\cdot{\bf
w},K}h_{K}||S^{-1}{\bf u}_{h}||_{K}
+\sum\limits_{\sigma\in\varepsilon_{K}}\Lambda_{{\bf
w},\sigma}h_{\sigma}^{1/2}||[\gamma_{{\bf t}_{\sigma}}(S^{-1}{\bf
u}_{h})]||_{\sigma}.
\end{equation}
From the definitions of $\beta_{*,K}$
and $c_{\omega_{\sigma}}$ in this theorem, we respectively apply LEMMAs \ref{efficiency jump}-\ref{efficiency volumn } to the above
  two inequalities  so as  to obtain
\begin{equation}\label{efficiency conforming and convection 1}
\eta_{NC,K}\leq
c_{36}\{\beta_{*,K}\mathcal{E}_{K}+\sum\limits_{\sigma\in\varepsilon_{K}}
c_{\omega_{\sigma}}\Lambda_{\sigma}^{1/2}||S^{-1/2}({\bf u}-{\bf
u}_{h})||_{\omega_{\sigma}}\}
\end{equation}
and
\begin{equation}\label{efficiency conforming convection 2}
\eta_{C,K}\leq c_{37}\{\beta_{*,K}\mathcal{E}_{K}+
\sum\limits_{\sigma\in\varepsilon_{K}}\Lambda_{{\bf
w},\sigma}c_{\omega_{\sigma}}||S^{-1/2}({\bf u}-{\bf
u}_{h})||_{\omega_{\sigma}}\}.
\end{equation}
The assertion (\ref{new 8}) follows from (\ref{efficiency conforming
and convection 1}) and (\ref{efficiency conforming convection 2}) by
taking $c_{4}:=\sqrt{2(d+1)}\max\{c_{36},c_{37}\}$.\\

{\bf Proof of THEOREM \ref{new 9}}: The local shape regularity of
elements implies
\begin{equation}\label{efficiency upwind estimator 1}
\eta_{U,K}\leq
\frac{c_{38}}{\sqrt{c_{S,K}}}\sum\limits_{\sigma\in\varepsilon_{K}}
|({\bf w}\cdot{\bf
n})|_{\sigma}|(h_{K}^{1/2}||\hat{\hat{p}}_{\sigma}||_{\sigma}+h_{\sigma}||S^{-1}{\bf
u}_{h}||_{\omega_{\sigma}}).
\end{equation}
Then the desired estimate (\ref{new 10}) follows from  (\ref{efficiency upwind estimator 1}), LEMMA
\ref{efficiency upwind-scheme} and the definitions of the constants
$\lambda_{\sigma}$, $\rho_{\sigma}$, and
$\mathcal{E}_{D,\omega_{\sigma}}$.

\section{Numerical experiments}
In this section, we test our proposed posteriori error estimators on
three model problems.

\subsection{Model problem with singularity at the origin}

We consider the problem (\ref{convection-diffufsion-equations1})
in an $L$-shape domain $\Omega=\{(-1,1)\times(0,1)\}\cup\{(-1,0)\times(-1,0)\}$ with ${\bf w}= r=0$ and $f=0$.
The exact solution  is given by
\begin{equation*}
p(\rho,\theta)=\rho^{2/3}\sin(2\theta/3),
\end{equation*}
where $\rho,\theta$ are the polar coordinates.
\begin{figure}[htbp]
  \begin{minipage}[t]{0.5\linewidth}
    \centering
    \includegraphics[width=2.5in]{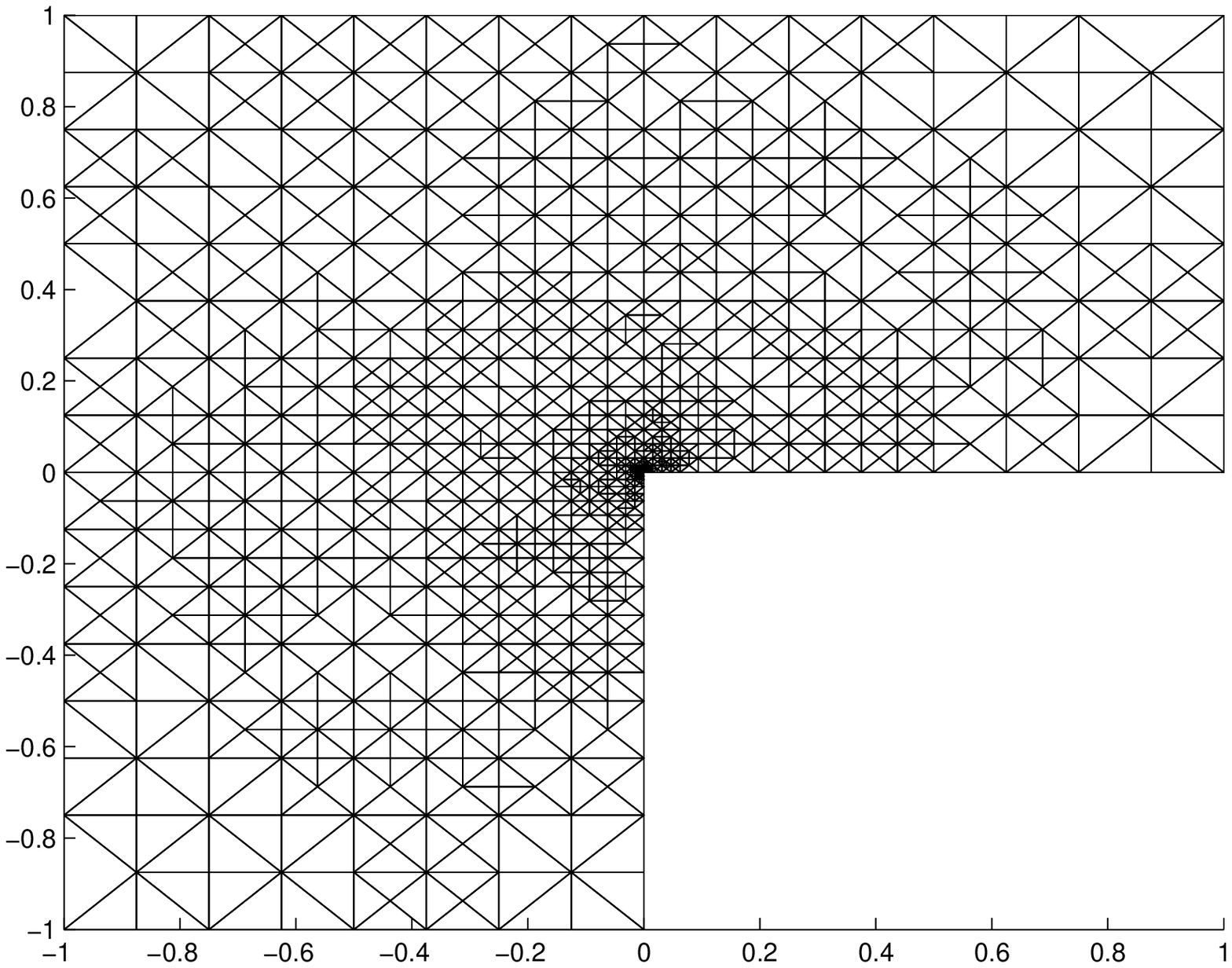}\\
  \end{minipage}
  \begin{minipage}[t]{0.5\linewidth}
    \centering
    \includegraphics[width=2.5in]{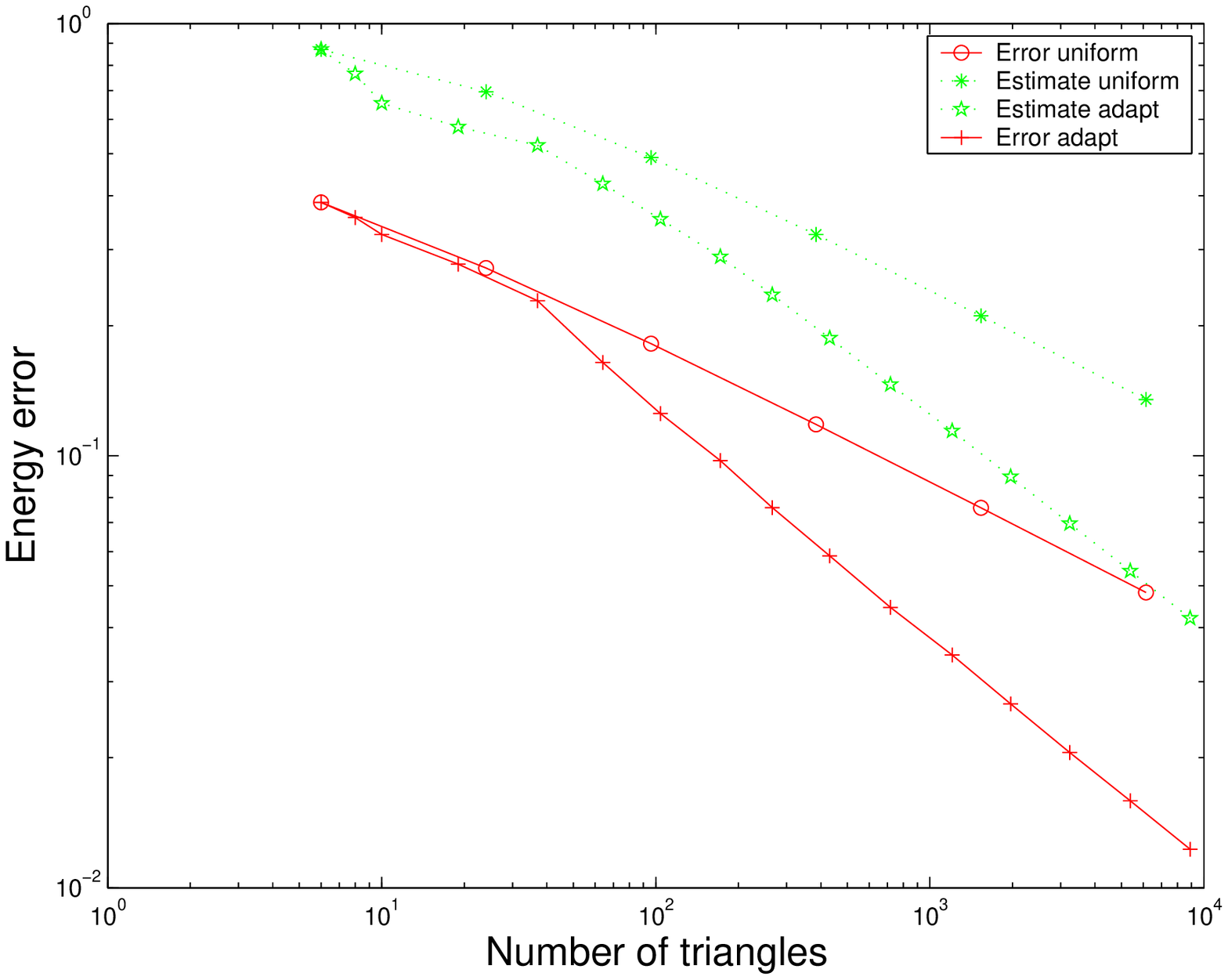}\\
  \end{minipage}
\addtocontents{lof}{figure}{FIG 8.1. {\small {\it A mesh with 1635
triangles
 (left) and the estimated and actual errors
  in uniformly / adaptively refined meshes (right).}}}\\
\end{figure}

It is well known that this model possesses singularity at the
origin. The original mesh consists of 6 right-angled triangles. We
employ the centered mixed scheme described in section 3.1 to compute
the approximaton solution, mark elements in terms of D\"{o}rfler
marking with the marking parameter $\theta=0.5$, and then  use the "longest edge" refinement to recover an
admissible mesh. Specially, the uniform refinement means that all
elements should be marked. We note that in the given case, the residual
estimators $\eta_{R,K}$ vanish over all $K\in\mathcal{T}_{h}$.

We see in the first figure of Fig 8.1 with 1635 elements that the
refinement concentrates around the origin, which means the predicted
error estimator captures well the singularity of the solution. The
second graph of Fig 8.1 reports the estimated and actual errors of
the numerical solutions on uniformly and adaptively refined meshes.
It can be seen that one can substantially reduce the number of
unknowns necessary to obtain the prescribed accuracy by using the a
posteriori error estimates and adaptively refined meshes, and that
the error of the flux in $L^{2}$ norm uniformly reduces  with a
fixed factor on two successive meshes, and that the adaptive mixed
finite element method is a contraction with respect to the energy
error.

\subsection{Model problem with inhomogeneous diffusion tensor\cite{Eigestad,Riviere;Wheeler2003,Vohralik1}}
We consider the problem
(\ref{convection-diffufsion-equations1}) in a square domain
$\Omega=(-1,1)\times(-1,1)$ with ${\bf w}=r=0$ and $f=0$, where $\Omega$ is divided into four subdomains
$\Omega_{i}$ ($i=1,2,3,4$) corresponding to the axis quadrants (in
the counterclockwise direction), and  the
diffusion-dispersion tensor $S$ is piecewise constant matrix with $S=s_{i}I$
in $\Omega_{i}$.  We suppose the exact solution of this model  has the form
\begin{equation*}
p(r,\theta)=r^{\alpha}(a_{i}sin(\alpha\theta)+b_{i}cos(\alpha\theta))
\end{equation*}
in each $\Omega_{i}$ with Dirichlet boundary conditions. Here $r,\theta$ are the polar coordinates in $\Omega$, $a_{i}$ and $b_{i}$ are constants
depending on $\Omega_{i}$, and $\alpha$ is a parameter. We note that the stress solution ${\bf u}=-S\nabla p$ is
not continuous across the interfaces, and only its normal component
  is continuous. It
finally exhibits a strong singularity at the origin.   We
consider two sets of coefficients in the following table:
\begin{center}
\scriptsize
\begin{tabular}{|c|c|} \hline
Case 1 &Case 2\\ \hline
$s_{1}=s_{3}=5$, $s_{2}=s_{4}=1$&$s_{1}=s_{3}=100$, $s_{2}=s_{4}=1$\\ \hline
$\alpha=0.53544095$ &$\alpha=0.12690207$\\ \hline
$a_{1}=\ \ 0.44721360$, $b_{1}=\ \ 1.00000000$&$a_{1}=\ \ 0.10000000$, $b_{1}=\ \ 1.00000000$\\
$a_{2}=-0.74535599$,  $b_{2}=\ \ 2.33333333$&$a_{2}=-9.60396040$, $b_{2}=\ \ 2.96039604$\\
$a_{3}=-0.94411759$,  $b_{3}=\ \ 0.55555555$&$a_{3}=-0.48035487$,  $b_{3}=-0.88275659$\\
$a_{4}=-2.40170264$, $b_{4}=-0.48148148$&$a_{4}=\ \ 7.70156488$, $b_{4}=-6.45646175$\\ \hline
\end{tabular}
\end{center}

The origin mesh consists of 8 right-angled triangles. We use the centered
scheme   compute the approximation solution, and mark elements in terms of D\"{o}rfler
marking with the marking parameter $\theta=0.7$  in the
first case and $\theta=0.94$  in the
second case.  We note that the elementwise
estimators $\xi_{K}$ are used as the a posteriori
error indicators, since the residual estimators $\eta_{R,K}$ vanish over $K\in\mathcal{T}_{h}$.

In Table 8.1 we show for Case 1  some results of  the actual error $E_{k}$, the a
posteriori indicator $\eta_{k}$,  the experimental convergence rate, ${\rm EOC}_{E}$, of  $E_{k}$, and the experimental convergence rate, ${\rm EOC}_{\eta}$,  of $\eta_k$, where
\begin{equation*}
{\rm EOC}_{E} :=\frac{\log(E_{k-1}/E_{k})}{\log({\rm DOF}_{k}/{\rm
DOF}_{k-1})},\ \ \
{\rm EOC}_{\eta} :=\frac{\log(\eta_{k-1}/\eta_{k})}{\log({\rm
DOF}_{k}/{\rm DOF}_{k-1})},
\end{equation*}
  and ${\rm DOF}_{k}$ denotes the number of elements  with
respect to the $k-$th iteration. We can see that the convergence
rates ${\rm EOC}_{E}$ and ${\rm EOC}_{\eta}$ are close to 0.5 as the
iteration number $k=15$, which means the optimal decay of the actual
error and a posteriori error indicator $\eta_{k}$ is almost attained
after 15 iterations with optimal meshes.

\begin{table}[!h]\renewcommand{\baselinestretch}{1.25}\small
\begin{center}
 \caption{Results of actual error $E_{k}$,  a
posteriori indicator $\eta_{k}$,  and their convergence rates  ${\rm EOC}_{E}$ and ${\rm EOC}_{\eta}$: Case 1
  }
\small 
\begin{tabular}{|c|c|c|c|c|c|} \hline
$k$& ${\rm DOF}_{k}$& $E_{k}$& $\eta_{k}$ &${\rm EOC}_{E}$& ${\rm
EOC}_{\eta}$\\ \hline 1&8&1.3665&5.0938&$-$&$-$\\ \hline
2&20&1.1346&3.4700&0.2030&0.4189\\ \hline
3&44&0.8682&2.9300&0.3394&0.2145\\ \hline
4&89&0.6672&2.5032&0.3738&0.2235\\ \hline
5&171&0.4953&2.0907&0.4562&0.2757\\ \hline
6&354&0.3708&1.7170&0.3979&0.2706\\ \hline
7&760&0.2751&1.5639&0.3907&0.1222\\ \hline
8&1368&0.2163&1.3529&0.4091&0.2466\\ \hline
9&2235&0.1776&1.1115&0.4016&0.4004\\ \hline
10&4025&0.1381&0.8958&0.4276&0.3667\\ \hline
11&7165&0.1106&0.7111&0.3851&0.4004\\ \hline
12&13188&0.0871&0.5566&0.3915&0.4015\\ \hline
13&24445&0.0671&0.4368&0.4227&0.3927\\ \hline
14&43785&0.0510&0.3365&0.4707&0.4476\\ \hline
15&76770&0.0387&0.2581&0.4915&0.4724\\ \hline
\end{tabular}
\end{center}
\end{table}

Fig 8.2  shows an adaptively refined mesh with 4763 elements and the
estimated and actual errors  against the number of elements in
adaptively refined meshes for Case 1.  Fig 8.3 shows an adaptively
refined mesh with 1093 elements and the actual error against the
number of elements in adaptively refined meshes for Case 2.

From the first figures of Fig 8.2-8.3, we can see that  the
refinement again concentrates around the origin, which means  the
adaptive mixed finite element method detects the region of rapid
variation. In the second graphs of Fig 8.1-8.3 each includes an
optimal convergence line, which shows in both cases, the energy
error performs a trend of descend with an optimal order convergent
rate. Simultaneously, from the second graphs of Fig 8.1-8.3, we also
see that the proposed estimators are efficient with respect to the
strongly discontinuously coefficients.

 We note that the energy error is approximated with a 7-point
quadrature formula in each triangle.

\begin{figure}[htbp]
  \begin{minipage}[t]{0.5\linewidth}
    \centering
    \includegraphics[width=2.25in]{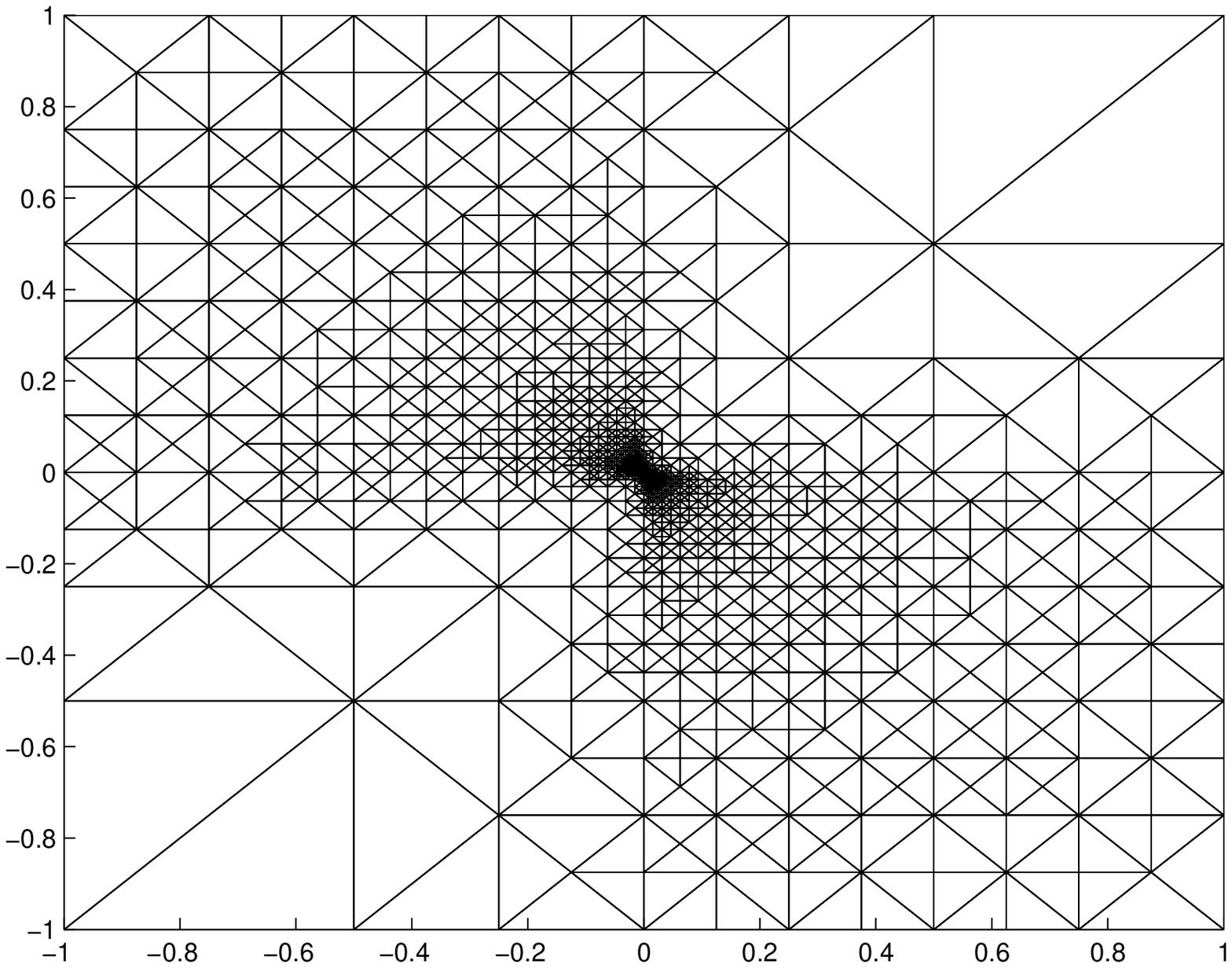}\\
  \end{minipage}
  \begin{minipage}[t]{0.5\linewidth}
    \centering
    \includegraphics[width=2.5in]{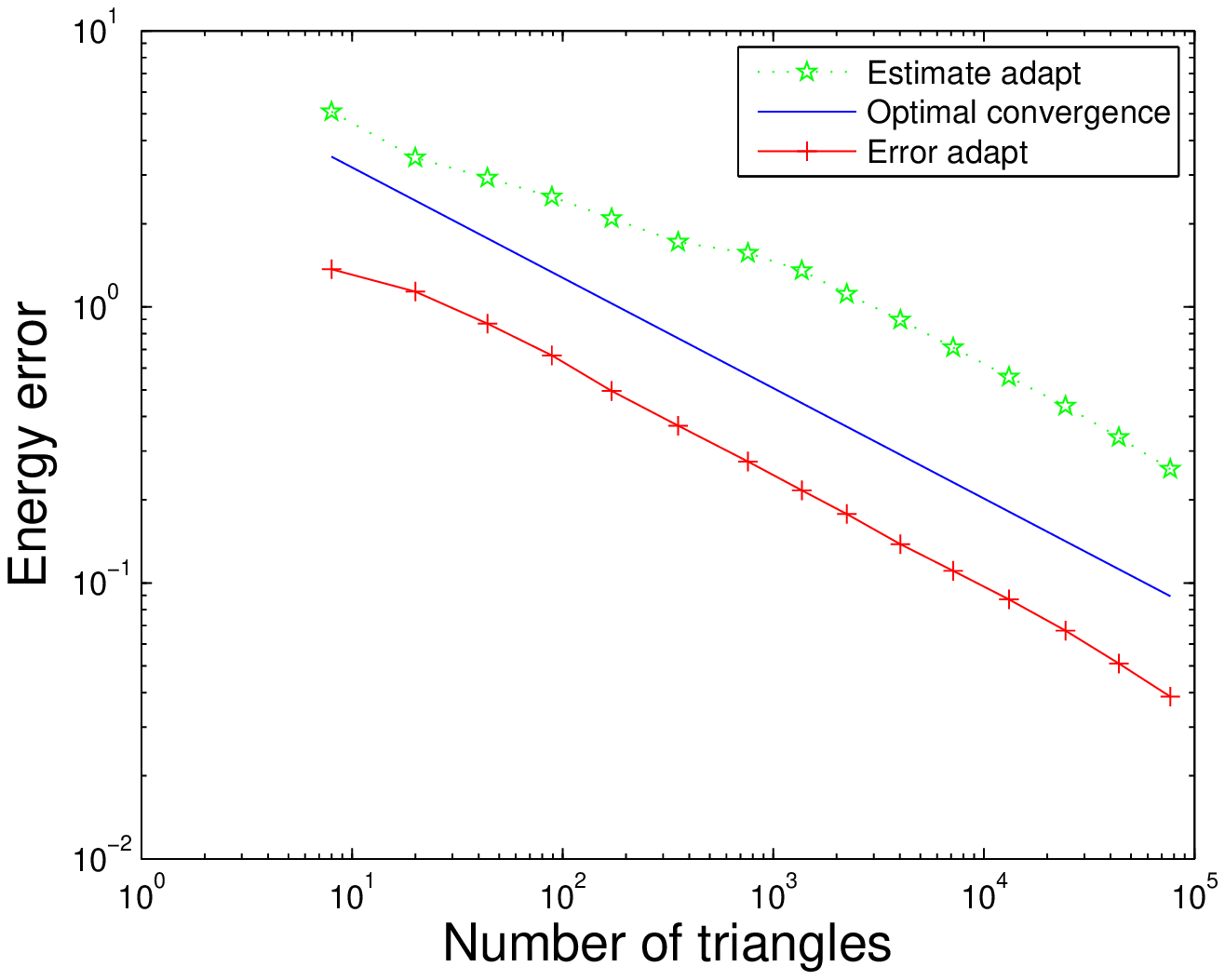}\\
  \end{minipage}
\addtocontents{lof}{figure}{FIG 8.2. {\small {\it A mesh with 4763
triangles
 (left) and the estimated and actual error
  against the number of elements in  adaptively refined meshes (right): Case 1.}}}\\
\end{figure}

\begin{figure}[htbp]
  \begin{minipage}[t]{0.5\linewidth}
    \centering
    \includegraphics[width=2.25in]{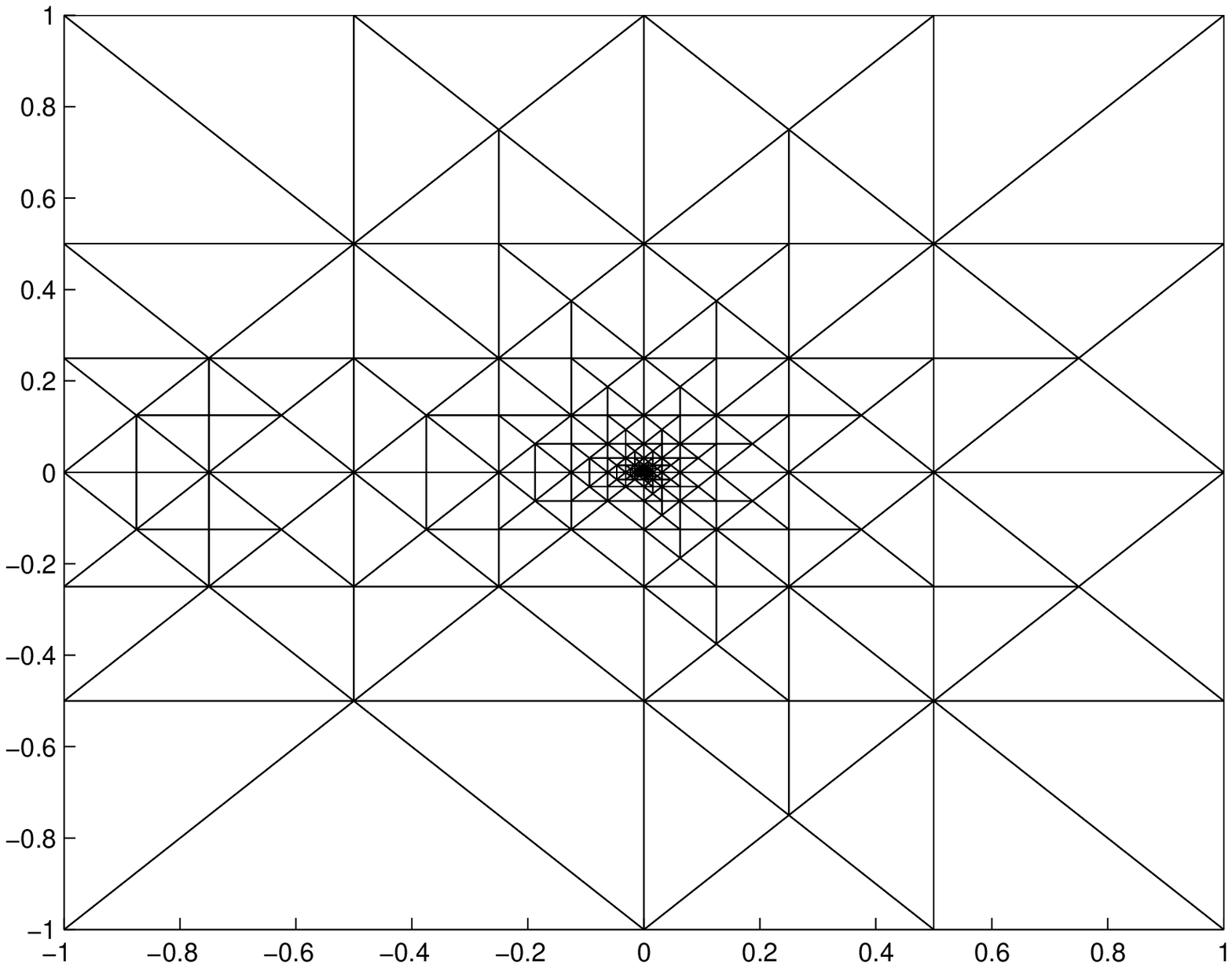}\\
  \end{minipage}
  \begin{minipage}[t]{0.5\linewidth}
    \centering
    \includegraphics[width=2.5in]{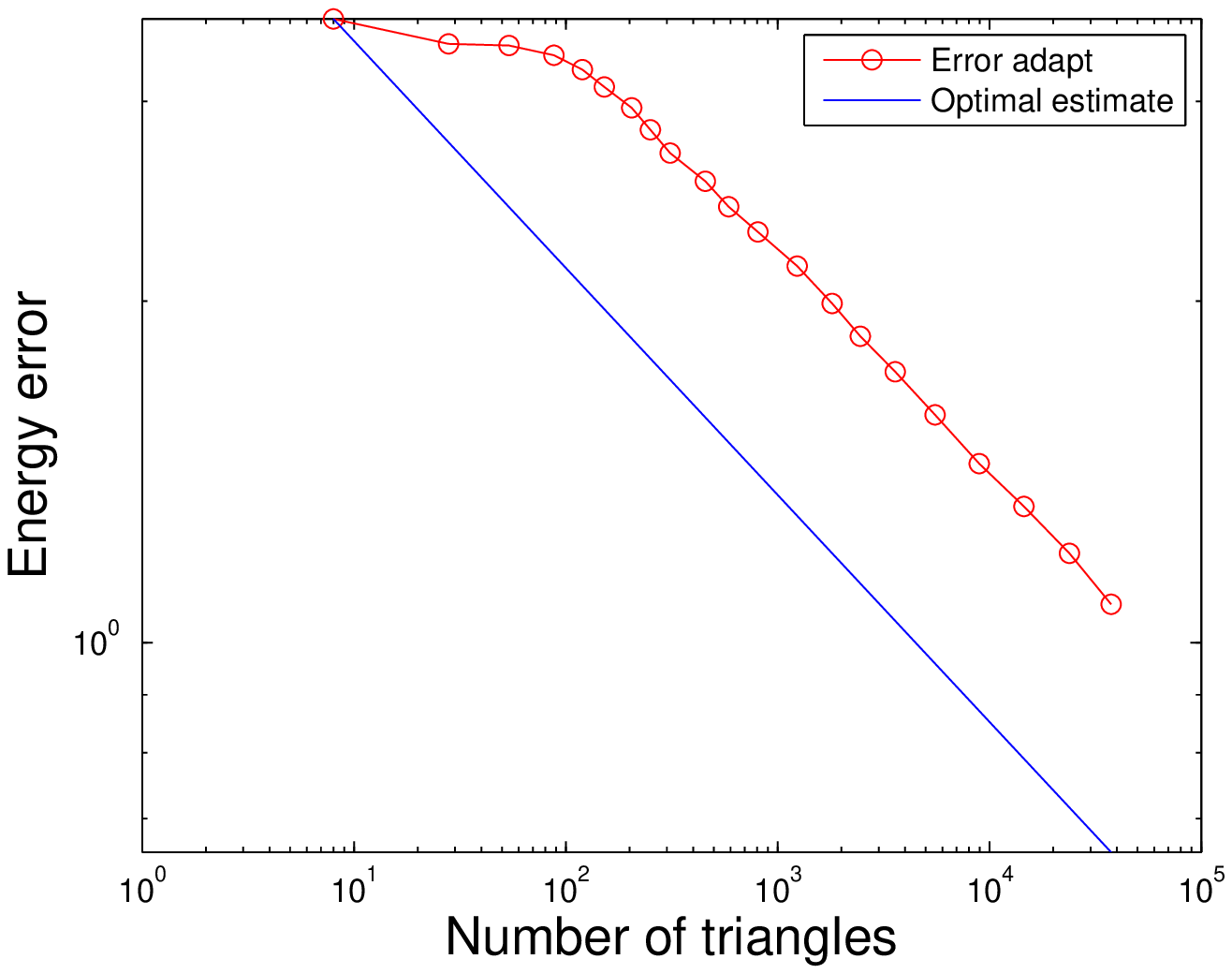}\\
  \end{minipage}
\addtocontents{lof}{figure}{FIG 8.3.\ {\small {\it A mesh with 1093
triangles
 (left) and the actual error
  against the number of elements in adaptively refined mesh (right): Case 2.}}}\\
\end{figure}

\subsection{ Convection-dominated model problem \cite{Vohralik1}}\
Let $S=\varepsilon I$, ${\bf w}=(0,1)$, $r=1$ and
$\Omega=(0,1)\times(0,1)$ in the model
(\ref{convection-diffufsion-equations1}). We consider four cases:
$\varepsilon =0.1, 0.01, 0.001, 0.0001$. Neumann boundary conditions
on the upper side,  Dirichlet boundary conditions elsewhere, and the
source term $f$ are chosen such that the exact solution has the form
\begin{equation*}
p(x,y)=0.5(1-\tanh(\frac{0.5-x}{a}))
\end{equation*}
with $a$ a positive constant. This solution is, in fact,
one-dimensional   and
 possesses an internal layer of width $a$ which we shall set,
 respectively, equal to 0.1, 0.05,  0.02,  0.001.

 We still start
 computations from an origin mesh which consists of 8 right-angled
 triangles, and refine it either uniformly (up to five refinements) or adaptively.
\begin{figure}[htbp]
  \begin{minipage}[t]{0.5\linewidth}
    \centering
    \includegraphics[width=2.25in]{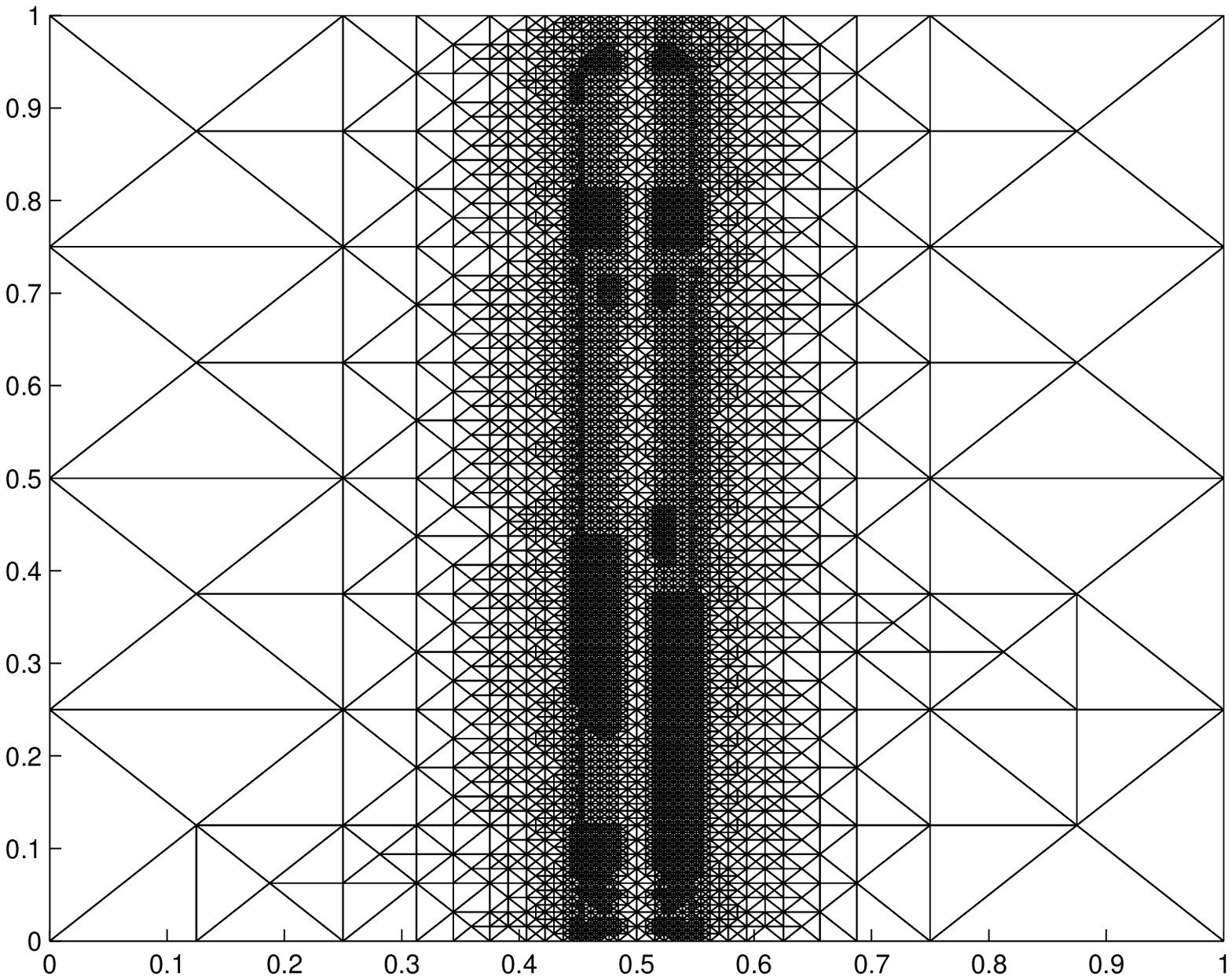}\\
  \end{minipage}
  \begin{minipage}[t]{0.5\linewidth}
    \centering
    \includegraphics[width=2.5in]{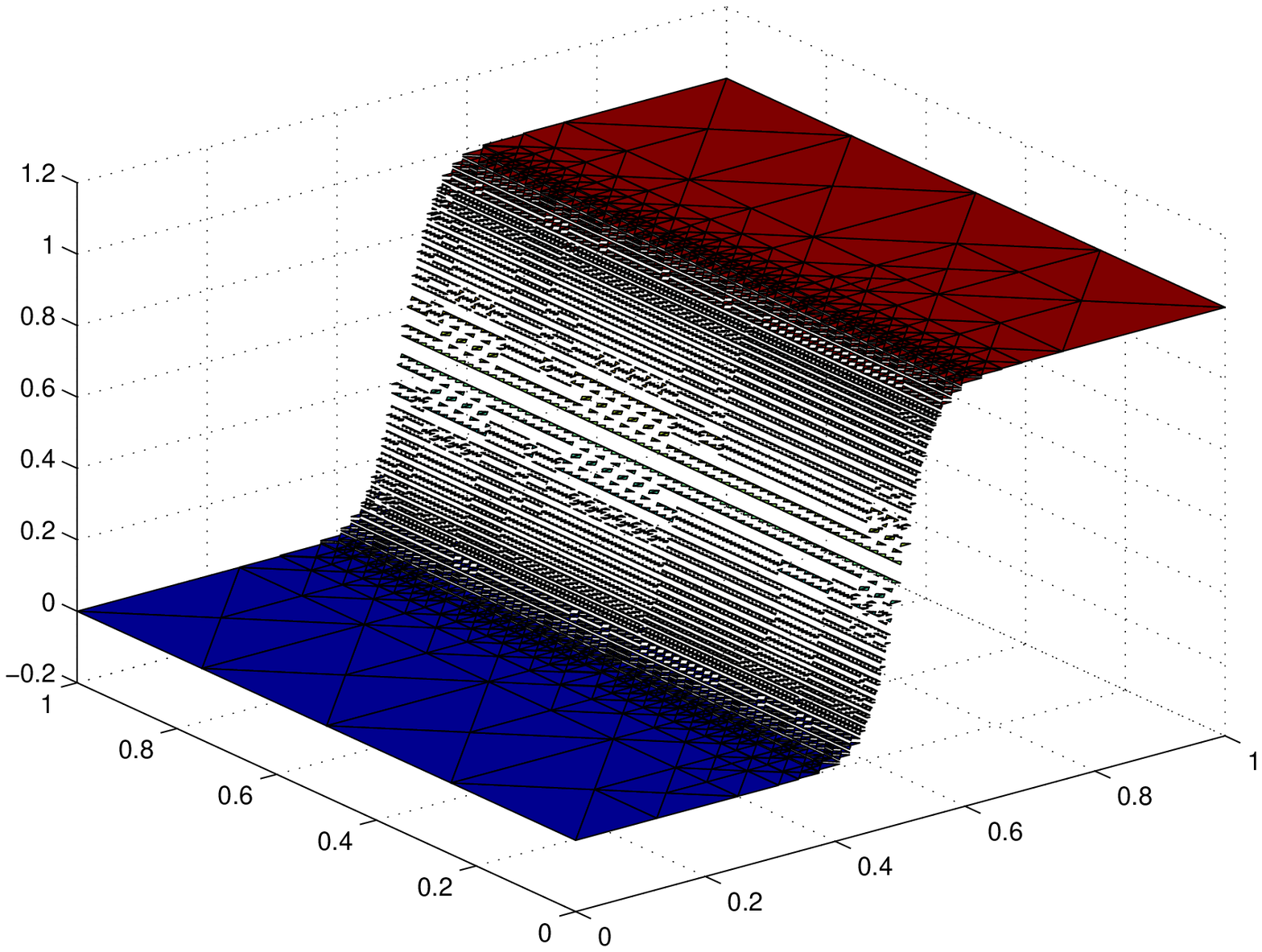}\\
  \end{minipage}
\addtocontents{lof}{figure}{FIG 8.4. {\small {\it A mesh with 12943
triangles
 (left) and the approximate displacement (piecewise constant) on the corresponding
 adaptively refined mesh (right) for $\varepsilon=0.01$ and a=0.05.}}}\\
\end{figure}
\begin{figure}[htbp]
  \begin{minipage}[t]{0.5\linewidth}
    \centering
    \includegraphics[width=2.25in]{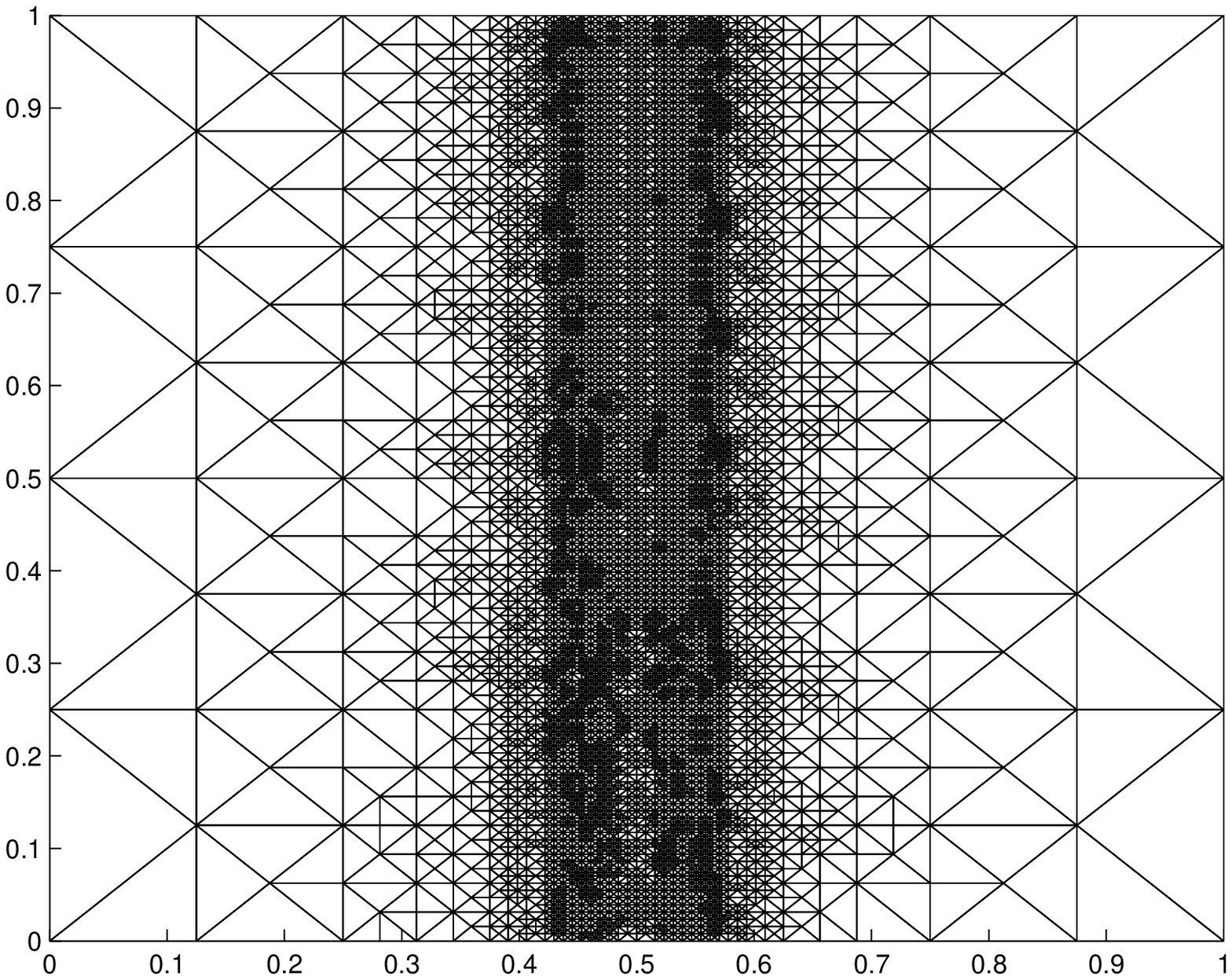}\\
  \end{minipage}
  \begin{minipage}[t]{0.5\linewidth}
    \centering
    \includegraphics[width=2.5in]{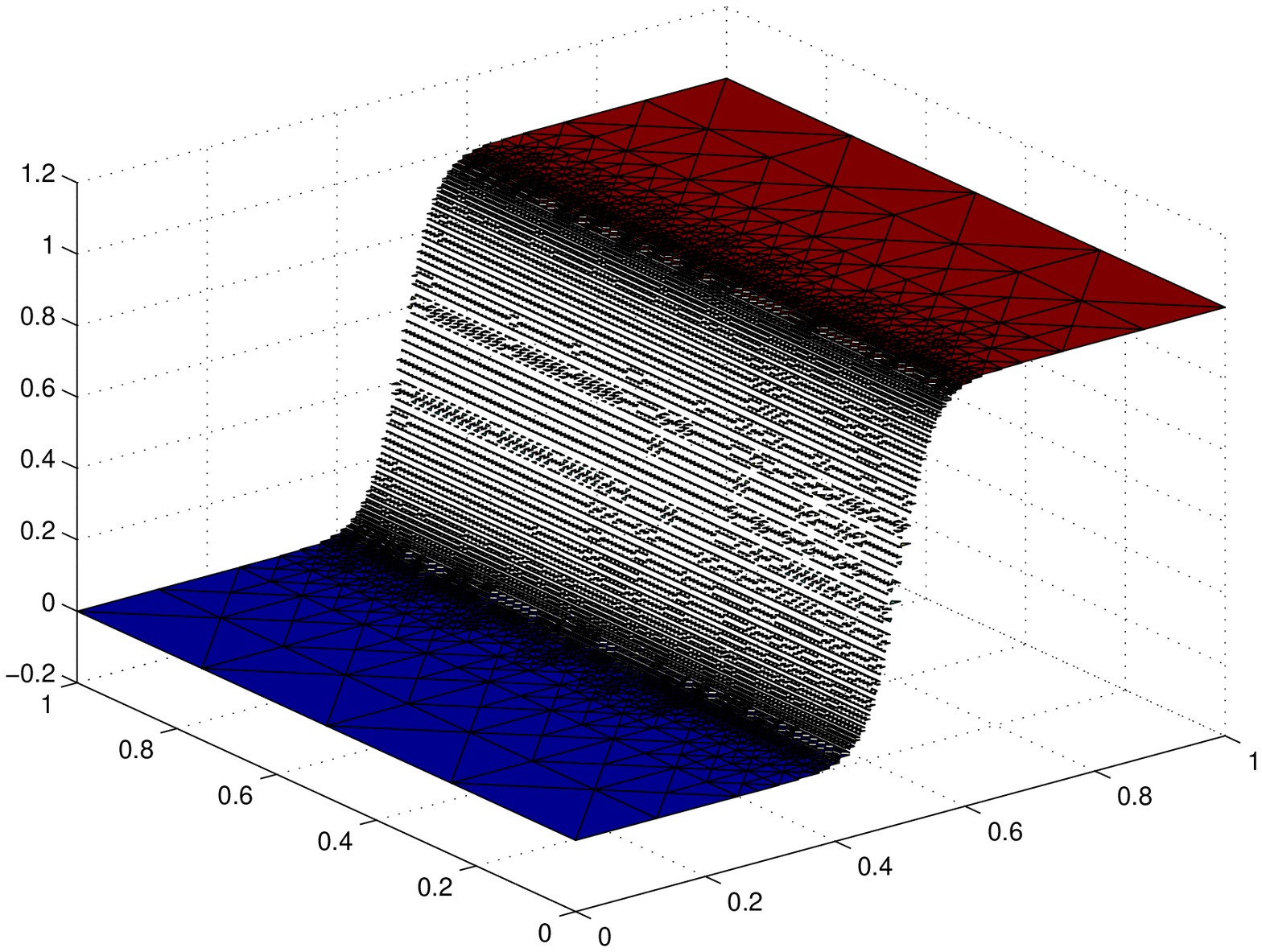}\\
  \end{minipage}
\addtocontents{lof}{figure}{FIG 8.5. {\small {\it A mesh with 16951
triangles
 (left) and approximate displacement (piecewise constant) on the corresponding
 adaptively refined mesh (right) for $\varepsilon=0.001$ and a=0.05.}}}\\
\end{figure}

In Fig 8.4 with $\varepsilon=0.01, a=0.05$ and Fig 8.5 with
$\varepsilon=0.001, a=0.05$, we can see that the refinement
concentrates at an internal layer of width $a=0.05$, and is away
from the center of the shock. Both the convection-dominated
regime on coarse grids and diffusion-dominated regime obtain the
progressive refinement. The effect is still rather good even if the approximation to
displacement is piecewise constant.
\begin{figure}[htbp]
  \begin{minipage}[t]{0.5\linewidth}
    \centering
    \includegraphics[width=2.25in]{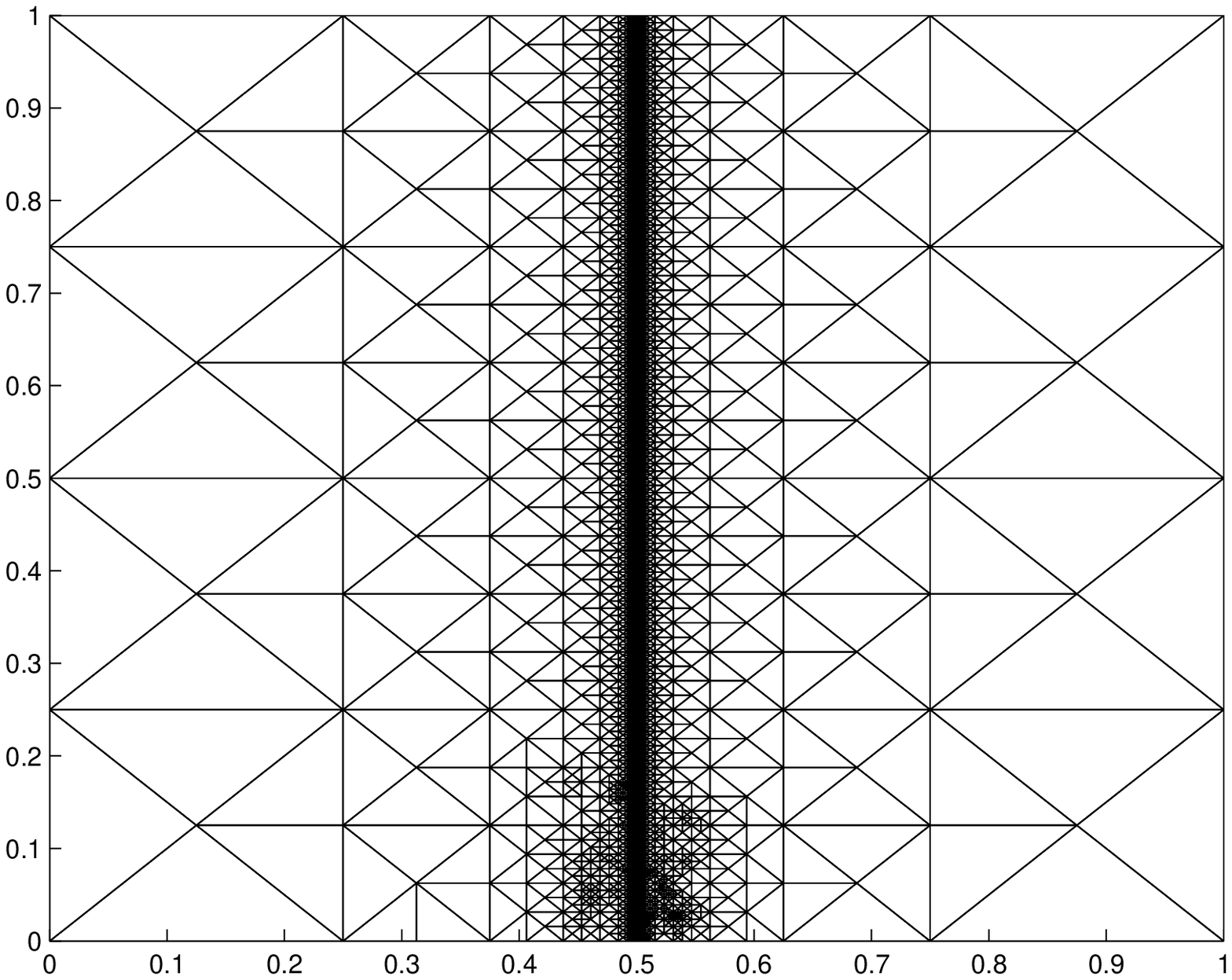}\\
  \end{minipage}
  \begin{minipage}[t]{0.5\linewidth}
    \centering
    \includegraphics[width=2.5in]{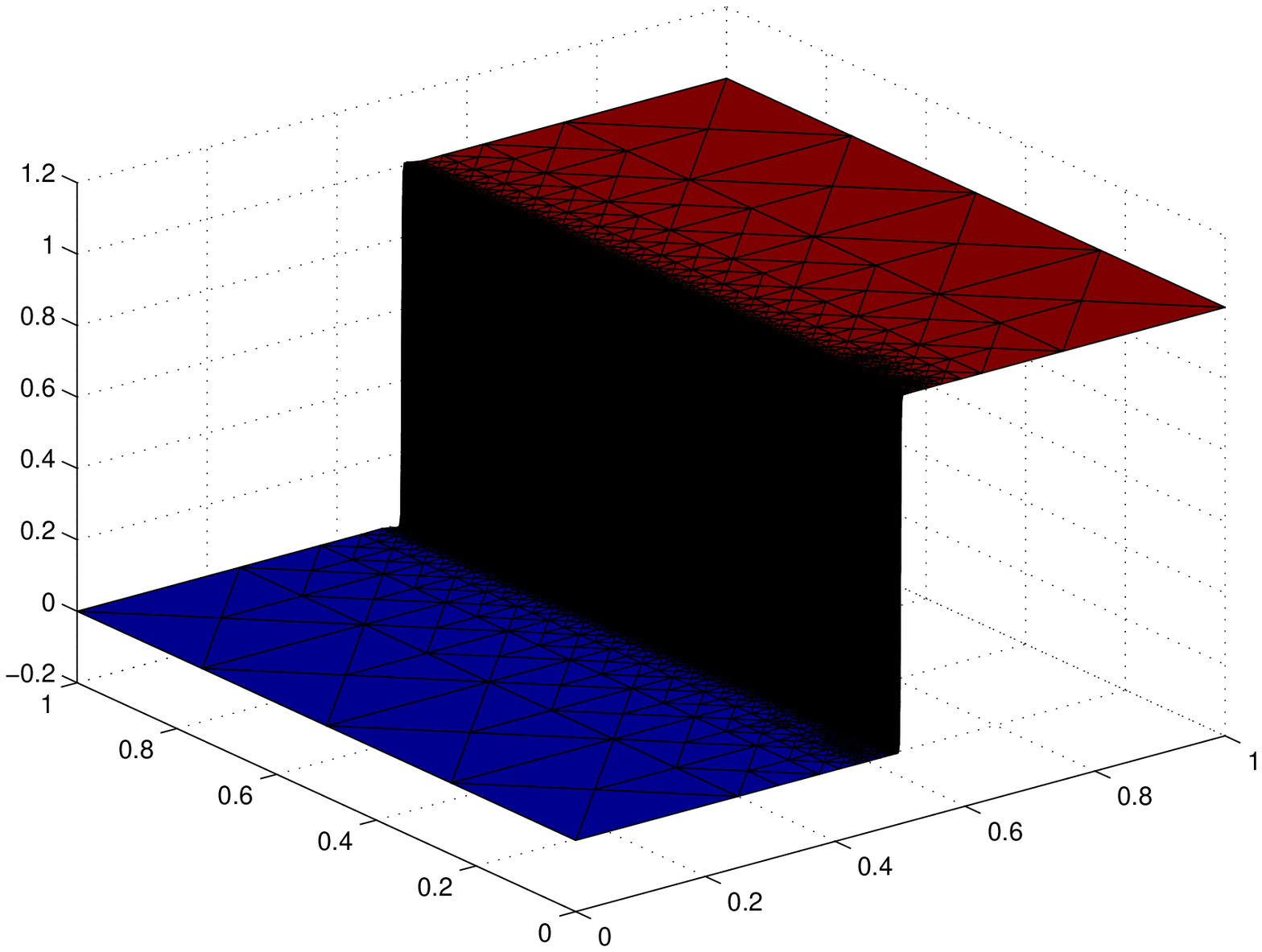}\\
  \end{minipage}
\addtocontents{lof}{figure}{FIG 8.6. {\small {\it A mesh with 39189
triangles
 (left) and postprocessing approximate displacement on the corresponding
 adaptively refined mesh (right) for $\varepsilon=0.0001$ and a=0.001.}}}\\
\end{figure}
\begin{figure}[htbp]
  \begin{minipage}[t]{0.5\linewidth}
    \centering
    \includegraphics[width=2.25in]{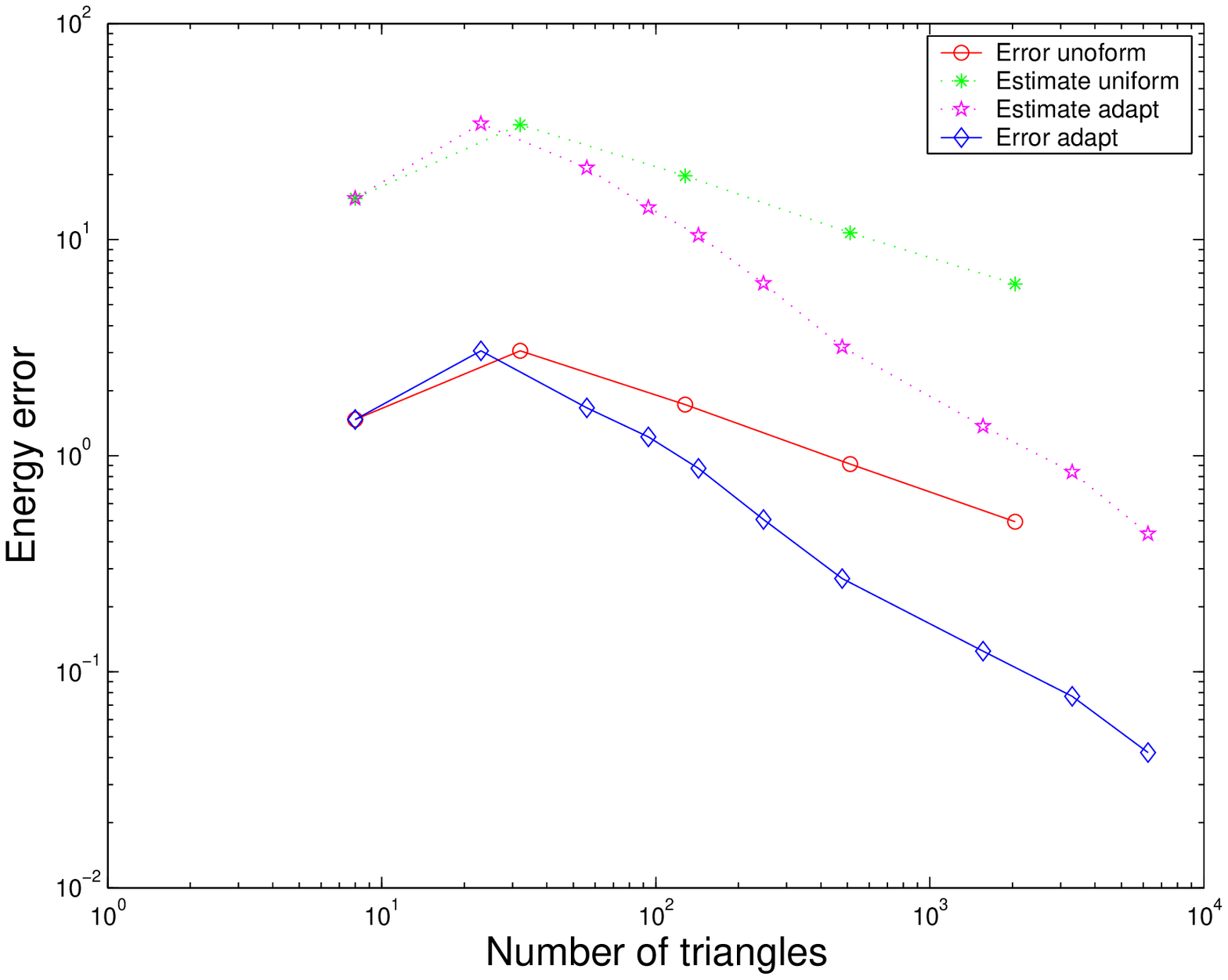}\\
  \end{minipage}
  \begin{minipage}[t]{0.5\linewidth}
    \centering
    \includegraphics[width=2.25in]{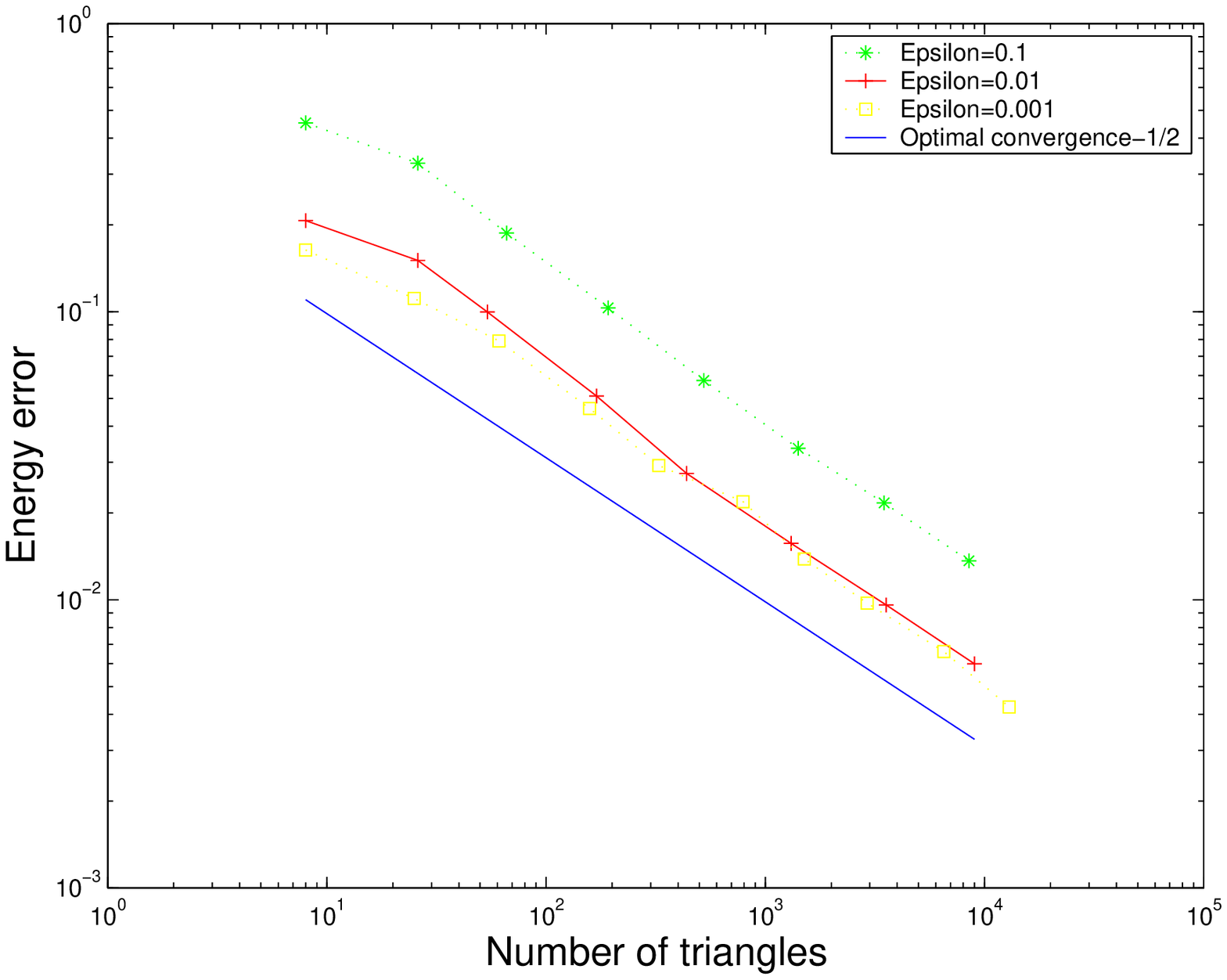}\\
  \end{minipage}
\addtocontents{lof}{figure}{FIG 8.7. {\small {\it Estimated and
actual error against the number of elements
  in uniformly and adaptively refined meshes for $\varepsilon=0.1, a=0.02$
  (left) and actual error against the nunber of elements in adaptively
  refined meshes for diffirent $\varepsilon$ for $a=0.1$(right) .}}}\\
\end{figure}

Fig 8.6 shows the mesh with 39184 triangles (left) and
postprocessing approximation to the scalar displacement on the
corresponding adaptively refined mesh (right) in case:
$\varepsilon=0.0001$ and width $a=0.001$. Here the value of the
postprocessing approximation on each node is taken as the
algorithmic mean of the values of the displacement finite element
solution on all the elements sharing the vertex. The reason for the
postprocessing is that the displacement finite element solution is
not continuous on each vertex of the triangulation. We again see
that the refinement focuses around layer of width $a=0.001$, this
indicates that the estimators actually capture interior layers and
resolve them in convection-domianed regions. In addition, the
postprocessing approximation to the scalar displacement obtains a
satisfactory result.

In Fig 8.7 with $\varepsilon=0.1,a=0.02$ (left), the estimated and
actual errors are plotted against the number of elements in
uniformly and adaptively refined meshes. Again, we see that one can
substantially reduce the unknowns necessary to attain the prescribed
precision by using the proposed estimators and adaptively refined
grids. The second graph of Fig 8.7 shows the actual error against
the number of elements in adaptively refined meshes for different
$\varepsilon$ in case $a=0.1$, and also concludes a line with
optimal convergence $-1/2$. In addition, we also see that the almost
same error decay occurs in cases: $\varepsilon=0.01$ and
$\varepsilon=0.001$.

\end{document}